\documentclass{article}

\usepackage[T1]{fontenc}

\usepackage{amsmath,amssymb,amsthm, xcolor, enumerate, tikz-cd, amsfonts, geometry, mathtools, xspace, authblk}
\usepackage[utf8]{inputenc}
\usepackage{graphicx} 
\usepackage[english]{babel}

\usepackage{hyperref}
\hypersetup{
     colorlinks=true,
     urlcolor=violet,
     }
\usepackage{titlesec}
\titleformat{\section}[block]{\small\bfseries\filcenter}{\thesection}{1em}{}
\titleformat{\subsection}[block]{\small\bfseries\filcenter}{\thesubsection}{1em}{}

\newtheorem{thm}{Theorem}[section]
\newtheorem{definition}[thm]{Definition}
\newtheorem{lem}[thm]{Lemma}
\newtheorem{prop}[thm]{Proposition}
\newtheorem{cor}[thm]{Corollary}

\newtheorem{rmk}[thm]{Remark}
\newtheorem{conj}[thm]{Conjecture}

\usepackage[nottoc]{tocbibind}

\newcommand{\Z}{\mathbb{Z}}
\newcommand{\CP}{\mathbb{C}\mathbb{P}^}
\newcommand{\R}{\mathbb{R}}
\newcommand{\C}{\mathbb{C}}

\newcommand{\ddb}{\sqrt{-1}\partial\bar\partial}
\newcommand{\Ima}{\text{Im }}
\newcommand{\al}{\alpha}
\newcommand{\dt}{\delta}
\newcommand{\ep}{\epsilon}
\DeclareMathOperator{\bl}{Bl}

\usepackage{lipsum}
\title{A gluing construction of constant scalar curvature K\"ahler metrics of Poincar\'e type}
\author{Yueqing Feng}
\date{}

\newcommand{\Addresses}{{
 \bigskip
\footnotesize
Yueqing Feng, \textsc{Department of Mathematics, University of California, Berkeley, CA 94720, USA.}\par\nopagebreak
\textit{E-mail address: }\href{fyq@berkeley.edu}{\texttt{fyq@berkeley.edu}}}}

\usepackage[
backend=biber,
    giveninits=true,
    style=numeric
]{biblatex}
\begin{document}

\maketitle

\begin{abstract}
   We consider a compact K\"ahler manifold admitting a constant scalar curvature K\"ahler metric and with no nontrivial holomorphic vector fields. After blowing up the manifold at finitely many points, we prove the existence of constant scalar curvature K\"ahler metrics on the complement of the exceptional divisors, with Poincar\'e-type singularities along them.
\end{abstract}

\section{Introduction}

In this article, we denote by $(X, \omega_X)$ a compact constant scalar curvature K\"ahler manifold of dimension $n$. This means the scalar curvature $s(\omega_X)$ of $\omega_X$ given by

$$s(\omega_X)\wedge\omega_X^n=n\cdot Ric(\omega_X)\wedge\omega_X^{n-1}$$

is constant. In this article, we construct constant scalar curvature K\"ahler manifold with Poincar\'e type singularities. The local model for the Poincar\'e type singularity is the standard Poincar\'e cusp metric

$$\omega_{\Delta^*}=\frac{\sqrt{-1}\,dz\wedge d\bar z}{(|z|\log|z|)^2}$$

on the punctured unit disk $\Delta^*=\{0<|z|<1\}$. Geometrically, for a simple normal crossings divisor $D$, Poincar\'e type K\"ahler metrics on $X\setminus D$ are complete K\"ahler metrics whose asymptotic behaviour near $D$ is locally modelled on $\omega_{\Delta^*}$ in the normal directions to $D$. For the precise definition used in this paper, see Definition~\ref{PT metric} in Section~\ref{section 2}.

The main theorem of this article is the following.

\begin{thm}\label{Main theorem}
    Let $X$ be a compact K\"ahler manifold of dimension $n$ with no non-trivial holomorphic vector fields, meaning that there is no non-zero holomorphic vector field that vanishes somewhere on $X$. Suppose $\omega_X$ is a constant scalar curvature K\"ahler (cscK) metric on $X$. Fix finitely many points $p_1, \cdots, p_k\in X$, let $Bl_{p_1, \cdots, p_k}X$ be the blow-up of $X$ at the chosen points, $\pi$ be the blow-down map, and $E_i\coloneqq \pi^{-1}(p_i)$ be the exceptional divisors. Then there exists a small constant $\epsilon_0>0$ such that for all $\epsilon \in (0, \epsilon_0)$, the manifold $Bl_{p_1, \cdots, p_k}X \setminus \bigcup_{i=1}^k E_i$ admits a cscK metric with Poincar\'e type singularity along $E_i$ within the given K\"ahler class 

     \begin{equation}\label{equation Kahler class}
        \pi^*[\omega_X]-\epsilon^2\displaystyle\sum_{i=1}^k[E_i].
    \end{equation} 
    
Here $[E_i]$ denotes the Poincar\'e dual of the $E_i$.
\end{thm}

We point out some concrete situations where Theorem~\ref{Main theorem} applies. The hypotheses are satisfied by any compact cscK manifold with discrete automorphism group. In particular, if $c_1(X)<0$, then the Aubin--Yau theorem gives a K\"ahler--Einstein metric of negative Ricci curvature, and the automorphism group is discrete. Thus compact K\"ahler manifolds with $c_1(X)<0$ provide a large class of examples to which the theorem applies. For further examples satisfying the same hypotheses, we refer the reader to the detailed discussion in \cite[Section~8.1]{arezzo2006blowing}.

The main theorem adds new examples to the existing list of cscK and extremal K\"ahler metrics of Poincar\'e type. We recall here some important examples: (a) the K\"ahler-Einstein metrics of Poincar\'e type, whose existence were shown by Tian--Yau \cite{TY87}, Kobayashi \cite{Kob84} and Cheng-Yau \cite{CY80}; (b) the toric extremal K\"ahler metrics of Poincar\'e type studied by Bryant \cite{Bry01} and Abreu \cite{Abr01}; (c) suitable blow-ups of (a) and (b) as discussed in \cite{Sek18}. 

The interest in canonical metrics of Poincar\'e type comes from several aspects:

\paragraph{Degeneration of smooth extremal K\"ahler metrics}

In complex dimension one, Poincar\'e cusp singularities arise naturally from the degeneration of hyperbolic Riemann surfaces \cite{wolpert1990hyperbolic,wolpert2007cusps,obitsu2008grafting}. In the Deligne--Mumford compactification, a degenerating family of compact Riemann surfaces of genus at least two may converge to a stable nodal curve, locally modeled near a node by the plumbing family $xy=t$. As $t\rightarrow0$, the degenerating hyperbolic metric near the node locally splits into complete finite-area hyperbolic metrics with cusp ends, each modeled on the Poincar\'e cusp metric $\omega_{\Delta^*}= \frac{ \sqrt{-1}dz\wedge d\bar{z}}{(|z|\log|z|)^2}$ on the punctured unit disc $\Delta^*$ in $\C$.

In higher dimensions, a related motivation comes from the compact cscK/extremal existence problem. In the toric setting, Donaldson \cite{Don02Toric} proposed that, when compact extremal metrics fail to exist, minimizing procedures such as the Mabuchi functional should lead to a decomposition of the moment polytope into stable pieces. The new facets carry zero boundary measure, and the corresponding pieces are expected to admit complete extremal K\"ahler metrics on complements of toric divisors. Apostolov--Auvray--Sektnan \cite{AAS17} studied the relation between these Donaldson complete toric metrics and Poincar\'e type asymptotics. Their results show that extremal Poincar\'e type metrics form a natural subclass of Donaldson metrics, and hence fit naturally into Donaldson's conjectural limiting picture for the compact extremal problem.

Poincar\'e type metrics form a natural subclass of Donaldson metrics
A concrete example of this phenomenon was given by Sz\'ekelyhidi \cite{szekelyhidi2009calabi}. For a ruled surface with a family of polarisations $\mathcal L_m$, he constructs a sequence of smooth extremal K\"ahler metrics whose degeneration yields complete extremal metrics of Poincar\'e type. These metrics arise from minimizing the Calabi functional within the momentum construction ansatz \cite{HS02}, along a path between stable and unstable polarisations.

\paragraph{Connections with conical metrics}   

Poincar\'e type metrics are complete and are connected to the smooth metrics by a family of incomplete conical metrics. Intuitively, as the cone angle approaches $2\pi$, we approach the smooth metrics; while when it approaches $0$, we approach the Poincar\'e type cuspidal singularity. On the punctured unit disc $\Delta^*$, the conical metrics $\omega_{\beta}\coloneqq  \frac{\beta^2 \sqrt{-1}dz\wedge d\bar{z}}{|z|^{2(1-\beta)}(1-|z|^{2\beta})^2}$ with cone angle $2\pi\beta$ have constant negative curvature, and converges pointwise to the Poincar\'e cusp metric $\omega_{\Delta^*}$. 

In higher dimension, Guenancia \cite{Gue15} generalizes this observation and shows that for a compact K\"ahler manifold $X$ and a divisor for which $K_X+D$ is ample, for a family of negative K\"ahler-Einstein metrics $\omega_{\beta}$ with cone angle $2\pi\beta$ along $D$, as $\beta$ goes to zero, $\omega_{\beta}$ converges to a Poincar\'e type metric $\omega_0$ in $C^{\infty}_{loc}(X\backslash D)$. In particular, $(X\backslash D, \omega_{\beta})$ converges to $(X\backslash D, \omega_0)$ in the pointwise Gromov-Hausdorff topology. 
 
\paragraph{Yau-Tian-Donaldson type conjecture for Poincar\'e type metrics}

First formulated by Sz\'ekelyhidi, later refined by Apostolov-Auvray-Sektnan, the Yau-Tian-Donaldson type conjecture for extremal metrics of Poincar\'e type says the following.

    \begin{conj}(Yau-Tian-Donaldson type conjecture for Poincar\'e type metrics \cite{Sze06}, \cite{AAS17}) \label{conj: YTD for PT}

    $(X, L, D)$ admits an extremal K\"ahler metric of Poincar\'e type on $X\backslash D$ if and only if 

    \begin{enumerate}[(i)]
        \item Relative Donaldson--Futaki invariant of any test configuration for the triple $(X, L, D)$ is non-negative, with equality holding if and only if the test configuration is product. 
        \item Sz\'ekelyhidi's numerical constraint holds.
        \item The restriction of extremal vector field of $X$ for $c_1(L)$, when restricted to $D$, is equal to the extremal vector field of $D$ in $c_1(L)|_D$.
    \end{enumerate}
    \end{conj}

The relationship between the relative K-stability (resp. K-stability) and the existence of extremal K\"ahler (resp. cscK) metric has been widely studied. Auvray \cite{Auv13-1, Auv17} proved the necessity of the Sz\'ekelyhidi numerical constraint and the extremal vector field condition. For $n=1$, given a log Riemann surface $(X, D= \sum_{i=1}^dp_i)$, \cite{sun2021projective} established a direct relationship between the complete cuspidal metric on $X\backslash D$ and the K-stability of the pair using quantization technique. Later, \cite{Aoi22} shows if there is no non-trivial holomorphic vector field on $D$ and the connected component of the identity in the holomorphic automorphism group preserving $D$ is trivial, then $K$-semistability of the pair is a necessary condition. In the toric surface setting, Apostolov--Auvray--Sektnan \cite{AAS17}
formulate a refined version of Conjecture~\ref{conj: YTD for PT} and verify it for all reduced
torus-invariant divisors on the Hirzebruch surfaces $\mathbb F_m$, $\mathbb{CP}^1\times\mathbb{CP}^1$, and $\mathbb{CP}^2$.

\medskip

There is another perspective of the main theorem. For fixed point $p\in X$, let $d_p(z)$ be the distance function from $z$ to $p$. Given $\lambda$ sufficiently large, we define the $(1, 1)$-form

\begin{equation}\label{reference metric}
   \omega_h\coloneqq \omega_X+\displaystyle\sqrt{-1}\partial\bar{\partial}(\log d_p^2(z) -\log(\lambda-\log d_p^2(z))).
\end{equation}

We say that $\omega_{PT}\coloneqq \omega_X+\sqrt{-1}\partial\bar{\partial}\varphi$ is a metric has Poincar\'e-Mok-Yau type singularity if it satisfies the two conditions in Definition~\ref{PT metric} with the sup of norms taken over $X\backslash \{p\}$. This definition can be easily extended to a compact manifold with finitely many points $p_1, \cdots, p_k$ removed, simply by choosing sufficiently large constants $\lambda_1, \cdots, \lambda_k$ such that 

$$\omega_h\coloneqq \omega_X+\ddb\displaystyle\sum_{i=1}^k(\log d_{p_i}^2(z)-\log(\lambda_i-\log d_{p_i}^2(z)))$$

is a positive $(1,1)$-form on $X\backslash\{p_1, \cdots, p_k\}$. This type of singularity is discussed in \cite{mokyau1983completeness, FYZ16}. In that paper, the Poincar\'e-Mok-Yau asymptotic is defined in a more general setting in the higher-codimensional complements of a manifold, and the definition requires only the quasi-isometric property stated in Definition~\ref{PT metric}.

\medskip

The main theorem thus implies the existence of constant scalar curvature K\"ahler metrics on the punctured manifold $X\backslash\{p_1, \cdots, p_k\}$ with complete Poincar\'e-Mok-Yau type singularity at each $p_i$. From this perspective, we compare the main theorem with the following proposition.

\begin{prop}(Gao-Yau-Zhou, \cite{GYZ})
   Given $X$ a compact K\"ahler manifold and $S\subset X$ a subvariety with codimension higher than or equal to two, then there are no complete K\"ahler-Einstein metrics on $X\backslash S$.  
\end{prop}

Then the main theorem implies that it is \textbf{\textit{not}} true in the cscK setting. 

\medskip

The scalar curvature of the resulting cscK metric in the main theorem is not the same as that of $\omega_X$, see Section \ref{subsection 9.1}. As an application, the main theorem produces scalar-flat K\"ahler metrics of Poincar\'e type.

\begin{cor}\label{cor SFK}
    Given a compact K\"ahler manifold $X$ of dimension $n$ with no non-trivial holomorphic vector fields, assume it is endowed with a scalar-flat K\"ahler(SFK) metric $\omega_X$. Assume $c_1(X)\ne 0$, given finitely many points $p_1, \cdots, p_k\in X$, then the manifold $Bl_{p_1, \cdots, p_k}X \setminus \bigcup_{i=1}^k E_i$ admits a scalar-flat K\"ahler metric with Poincar\'e type singularity along $E_i$ within the given K\"ahler class~(\ref{equation Kahler class}).
\end{cor}

We compare the main theorem to the blow-up constructions of Arezzo-Pacard \cite{arezzo2006blowing} and Sz\'ekelyhidi \cite{Sze14}. To obtain the cscK metric of Poincar\'e type in Theorem~\ref{Main theorem}, we carry out a gluing construction. More specifically, instead of gluing the Burns--Simanca scalar-flat K\"ahler metric on the total space of the line bundle $\mathcal{O}(-1)$ over $\CP{n-1}$, here we glue in a complete scalar-flat K\"ahler(SFK) metric on $\mathcal{O}(-1)\backslash E\cong Bl_0\C^n\backslash E\cong\C^n-\{0\}$, which arises from the momentum construction of Hwang-Singer \cite{HS02}. We refer to this metric as \textbf{\textit{the Hwang-Singer metric}}, denoted by $\omega_{HS}$. We will show that $\omega_{HS}$ has Poincar\'e type singularity along the divisor $E$. 

\medskip

With additional assumptions, we extend the result to those $X$ that admit non-trivial holomorphic vector fields. 

\begin{thm}\label{thm application 1}
    Given a compact K\"ahler manifold $X$ and a maximal compact subgroup $K$ of $G=Aut(X)$, let $\mathfrak{k}$ be its Lie algebra. Assume $X$ admits a cscK metric $\omega_X$ invariant under $K$. Let $\mu: X\rightarrow\mathfrak k^*$ be a normalized moment map with respect to $\omega_X$. Choose points $p_1, \cdots, p_k$ and $a_1, \cdots, a_k>0$ such that the genericity condition $\mu(p_1), \cdots, \mu(p_k) \text{ span }\mathfrak{k}$ and the balancing condition $\displaystyle\sum_{i=1}^ka_i\mu(p_i)=0$ are satisfied. Assume also that the extremal vector field $\mathcal X_{\ep}$ determined by the K\"ahler class

    \begin{equation}\label{Kahler class with ai}
        \pi^*[\omega_X]-\ep^2\displaystyle\sum_{i=1}^ka_i[E_i]
    \end{equation}
    
    and the subgroup $K$ satisfies

     \begin{equation}\label{theorem condition}
        \mathcal X_{\ep}|_{E_i}=0, i=1, \cdots, k.
    \end{equation}  
    
    Then $Bl_{p_1, \cdots, p_k}X\backslash \bigcup_{i=1}^k E_i$ admits a cscK metric with Poincar\'e type singularity along $E_i$ within the K\"ahler class (\ref{Kahler class with ai}).
\end{thm}


The Poincar\'e type extremal vector field is defined using the Futaki character introduced in \cite{Auv13-2}, see Section~\ref{subsection 9.3}. It depends only on the K\"ahler class and the maximal compact subgroup $K$. Note that if a cscK Poincar\'e type metric exists in the given class, then $\mathcal X_\epsilon=0$. In Theorem~\ref{thm application 1}, this global vanishing is not assumed in advance. The hypothesis used in the gluing argument is the weaker divisor condition $\mathcal X_\epsilon|_{E_i}=0.$

The above results are all concerned with cscK metrics. In fact, the arguments can be extended to Poincar\'e type extremal K\"ahler metrics as well. 

\begin{thm}\label{thm application 2}
    Given a compact K\"ahler manifold $X$ endowed with an extremal K\"ahler(extK) metric $\omega_X$. Assume $\omega_X$ is invariant under the action of the maximal compact subgroup $K$ of $G=Aut(X)$ containing its extremal vector field. Let $T\subset K$ be a compact torus and $H$ be the centralizer of $T$ in $K$. Use $\mathfrak t, \mathfrak k, \mathfrak h$ to denote the corresponding lie algebras and $\mu$ to denote a normalized moment map. Choose points $p_1, \cdots, p_k$ on $X$ and $a_1, \cdots, a_k>0$ such that the extremal vector field vanishes at each point, and further, the balancing condition
    $\displaystyle\sum_{i=1}^k a_i^{n-1}\mu(p_i)\in\mathfrak t$ and the genericity condition $\mathfrak t+\langle \mu(p_1), \cdots, \mu(p_k)\rangle=\mathfrak h$ are satisfied. Assume also the extremal vector field $X_{\ep}$ satisfies
    
    $$\mathcal X_{\ep}|_{E_i}=0, i=1, \cdots, k.$$
    
    Then $Bl_{p_1, \cdots, p_k}X\backslash \bigcup_{i=1}^k E_i$ admits an extremal K\"ahler metric with Poincar\'e type singularity along $E_i$ in the K\"ahler class (\ref{Kahler class with ai}).
\end{thm}

The conditions on the choice of points are the same as those required in the blowing up situation in \cite{APS11}. The vector field condition, again seen from \cite{Auv14} Theorem 4, is a necessary condition for the existence of extremal K\"ahler metric of Poincar\'e type in the given K\"ahler class (\ref{Kahler class with ai}).

\medskip


In Section~\ref{section 2}, we also show that there exists \textbf{\textit{a family of conical metrics}} with cone angle $\beta\in(0, 1)$ along the zero section connecting the classical Burns--Simanca metric with the cuspidal metric $\omega_{HS}$. If we use this family to do the gluing construction, then we expect to obtain a family of conical metrics on $Bl_{p_1, \cdots, p_k}X \setminus \bigcup_{i=1}^k E_i$ connecting the smooth metric constructed in \cite{arezzo2006blowing} by Arezzo-Pacard  and our Poincar\'e type metric in the main theorem. 

\paragraph{Strategy of the proof and Outline of the article} 

\begin{itemize}
    \item In Section \ref{section 2}, we discuss the derivation of $\omega_{HS}$ under the momentum construction and give an expression of $\omega_{HS}$ under the LeBrun ansatz. Besides, we introduce the conical family of scalar-flat K\"ahler metrics connecting the Burns--Simanca metric and the cuspidal metric $\omega_{HS}$. Further analysis on the expansion of $\omega_{HS}$ near the asymptotically Euclidean (AE) end and near the divisor are presented in Section \ref{section 3}. 

    \item We restrict to the situation of blowing up one point $p$. In Section \ref{section 4}, we construct a candidate metric on $Bl_pX\backslash E$ and set up the cscK equation we aim at solving. In Section \ref{section 5} we introduce the weighted H\"older norm space to study the behavior of functions near the Poincar\'e type cusp and define doubly weighted spaces for functions on $Bl_0\C^n\backslash E$. 
    
    \item The main technical difficulties are addressed in Section~\ref{section 7}. We construct a surjective map between H\"older spaces by explicitly constructing its inverse. The strategy is to first decompose the function into parts supported on the AE end and the cusp end, where local inverses are constructed separately. These are then glued together to obtain a global inverse. On the cusp end, we provide a new approach that bypasses the gap in the Fredholm-index argument of \cite{Sek18}, as pointed out in \cite{xuzheng24}. Using the geometry of the Hwang-Singer metric near the exceptional divisor and the Fourier decomposition along the $S^1$-action, we separate the $S^1$-invariant component from the nonzero modes and treat them separately. The key estimate is provided by Lemma~\ref{lemma essential inequality}.
    
    \item In Section \ref{section 8}, we need to solve a modified cscK equation due to the complexity introduced by the surjective map constructed in Section \ref{section 7}. Then we use the assumption that $X$ has no non-trivial holomorphic vector field to show this glued metric on $Bl_pX\backslash E$ is indeed cscK.

    \item In Section \ref{section 9}, we give an explicit formula for the constant scalar curvature on $Bl_pX\backslash E$ in Equation~\eqref{expansion of scalar curvature of solution} and give a proof of Corollary~\ref{cor SFK}, see \ref{subsection 9.1}. We prove Theorem~\ref{thm application 1} in \ref{subsection 9.2} and Theorem~\ref{thm application 2} in \ref{subsection 9.3}. 
\end{itemize}

\paragraph{Notations of weight near Poincar\'e type singularity}

The Definition~\ref{def weighted norm} of weight for H\"older spaces near Poincar\'e type cusp may differ by a sign in different contexts. To avoid confusion, we mention here that the sign we choose aligns with the one chosen by Auvray in \cite{Auv13-1, Auv13-2, Auv14, Auv17} and by Lockhart-McOwen in \cite{LM85}. It differs by a negative sign from the one chosen by Sektnan in \cite{Sek16, Sek18}, by Pacard in \cite{Pac08} and by Xu-Zheng in \cite{xuzheng24}.

\paragraph{Acknowledgements}

The author is sincerely grateful to Song Sun for many insightful discussions and constant encouragement. The author thanks Lars Martin Sektnan, Xuwen Zhu, Junsheng Zhang, Yifan Chen, and Daigo Ito for helpful discussions, and Yulun Xu for valuable conversations and for pointing out an issue in an earlier version. The author also thanks the anonymous referee for many constructive suggestions that improved the paper. The author is grateful to IASM at Zhejiang University for their hospitality during the visit when part of this research was carried out, and to the NSF for their generous support through grant DMS-2304692.

\section{Momentum construction of scalar-flat K\"ahler metric}\label{section 2}

In this section, we give the precise definition of Poincar\'e type singularity and discuss how the momentum construction in \cite{HS02} gives rise to the cuspidal scalar-flat K\"ahler metric on the punctured line bundle $\mathcal{O}(-1)^*$ appearing in the introduction. 

\begin{definition} (Poincar\'e type K\"ahler metric, \cite{Auv13-1, Auv17})\label{PT metric}
    Given $(X, \omega_X)$ a compact complex manifold and $D$ a smooth divisor in $X$ with $\sigma\in H^0(X, \mathcal{O}(D))$ being a holomorphic defining section. Fix $\lambda$ a sufficiently large constant such that
    
    $$\omega_h\coloneqq \omega_X-\displaystyle\sqrt{-1}\partial\bar{\partial}\log(\lambda-\log(|\sigma|^2))$$

    is a positive $(1, 1)$-form on $X\backslash D$. We say a closed, smooth $(1,1)$-form  
    
    $$\omega_{PT}\coloneqq \omega_X+\sqrt{-1}\partial\bar{\partial}\varphi$$
    
    on $X\backslash D$ is a \textbf{Poincar\'e type K\"ahler metric} if 
    \begin{itemize}
        \item $\omega_{PT}$ is \textit{quasi-isometric} to $\omega_{h}$, which means there exists some $C>0$, such that $\frac{1}{C}\omega_h\le\omega_{PT}\le C\omega_{h}$ and $\forall i\ge 1$, $\sup_{X\backslash D}|\nabla^i_{\omega_h}\omega_{PT}|<\infty$;
        \item for $h\coloneqq \log(\lambda-\log(|\sigma|^2))$, $\varphi$ is a smooth function on $X\backslash D$ with $\varphi=O(h)$, and $\forall i\ge 1$, $\sup_{X\backslash D}|\nabla^i_{\omega_{h}}\varphi|<\infty$.
    \end{itemize}
\end{definition}

Consider a point on $D$ and take a neighborhood $U$ under coordinates $(z_1, \cdots, z_n)$ such that $U\cap D=\{z_1=0\}$, then near the divisor, the reference model metric $\omega_h$ is asymptotic to 

$$\displaystyle\frac{\sqrt{-1}dz_1\wedge d\bar{z}_1}{|z_1|^2\log^2|z_1|^2}+\displaystyle\sum_{i=2}^n\sqrt{-1}dz_i\wedge d\bar{z}_i.$$

Near every point on the divisor, the Poincar\'e type metric is then modeled on the product of the Poincar\'e metric on the punctured disc with a smooth metric on the divisor. In particular, metrics asymptotic to the product metric above near $D$ are called asymptotically hyperbolic metrics. This is a special class of Poincar\'e type metrics. More generally, as discussed in \cite{Auv14, Auv17, Auv13-1}, Poincar\'e type metrics are also defined in the case where $D$ is a simple normal crossing divisor.

We now proceed with the momentum construction. Let $p: (L, h)\rightarrow (M^m,\omega_M)$ be a Hermitian holomorphic line bundle with curvature form $\gamma=-\displaystyle\sqrt{-1}\partial\bar{\partial}\log h$, where $M^m$ is a K\"ahler manifold of dimension $m$. For $r=\log|\cdot|_h$ the log of the fiberwise norm function defined by $h$, consider the Calabi ansatz

\begin{equation}\label{calabi ansatz}
    \omega=p^*\omega_M+2\displaystyle\sqrt{-1}\partial\bar{\partial}f(r).
\end{equation}

Then consider the natural $S^1$-action on $L$, and we use $\xi$ to denote its generator. Let $\tau$ be the corresponding moment map determined by $\iota_{\xi}\omega=-d\tau$, then $|\xi|_h^2$ is constant on the level sets of $\tau$, which implies the existence of a function $w: I\rightarrow (0, \infty)$ satisfying $w=|\xi|_h^{-2}$. The inverse of $w$, denoted by $\varphi$, is called the \textit{momentum profile} of $\omega_M$. Then we rewrite the metric \ref{calabi ansatz} as

$$\omega_{\varphi}=p^*\omega_M+2\ddb\varphi(\tau).$$

We choose the interval $I$ such that the closed $(1,1)$-form 

$$\omega_M(\tau)\coloneqq \omega_M-\tau\gamma$$

is positive for all $\tau\in I$. Let $g_M(\tau)$ and $Scal(g_M(\tau))$ denote the associated metric and the scalar curvature, respectively. From \cite{HS02}, we know the existence of complete cscK metrics on (punctured) line/disc bundles:

\begin{thm}(\cite{HS02}, Theorem C)\label{Thm C}
    Let $p: (L, h)\rightarrow (M^m,\omega_M)$ be the holomorphic line bundle and $I$ be the interval. Assume $\gamma\le 0$, $\gamma\ne 0$ and $g_M$ is complete. Then there exists a real number $c_0\in\mathbb{R}$ such that for any $c<c_0$ and finitely many $c>c_0$, the punctured disc bundle, denoted by $\Delta^*(L)$, carries a complete metric $g_c$ with constant scalar curvature $c$. If $c=c_0$, the metric extends to a complete metric on the punctured line bundle $L^*$.
\end{thm}

From the proof, we know the completeness of the metric near $\tau=0$ implies that $\varphi$ vanishes to order two, i.e., $\varphi(0)=\varphi'(0)=0$. Define $Q: I\times M^m\rightarrow\mathbb{R}$ to be

$$Q(\tau)\coloneqq\displaystyle\frac{\omega_M(\tau)^m}{\omega_M^m}.$$

In our case, take

$$M=\mathbb{C}\mathbb{P}^1, \quad L_0\coloneqq \mathcal{O}(-1), \quad L\coloneqq kL_0=\mathcal{O}(-k), \omega_M=2\omega_{FS},$$ 

then 

$$I=(0, \infty), \quad \gamma=-\frac{k}{2}\omega_{FS},  \quad Q(\tau)=1+\frac{k\tau}{2},  \quad Scal(g_M(\tau))=\displaystyle\frac{1}{1+\frac{k\tau}{2}}.$$

Here on $\CP1$, $\omega_{FS}$ satisfies Ric $\omega_{FS}=2\omega_{FS}$ and in homogeneous coordinate $Z$ can be written as $\ddb\log|Z|^2$. To obtain the explicit expression of momentum profile, first, we recall from \cite{HS02} the expression of $Scal(g_M(\tau))$:

\begin{thm}(\cite{HS02}, Theorem A)
     Starting from the above data the holomorphic line bundle $p: (L, h)\rightarrow (M^m,\omega_M)$ and interval $I$, let $g_{\varphi}$ be the metric associated to $\omega_{\varphi}$, then

    \begin{equation}\label{HS}
    Scal(g_{\varphi})=Scal(g_M(\tau))-\frac{1}{2Q}\displaystyle\frac{\partial^2}{\partial \tau^2}\left(Q\varphi\right)(\tau).  
    \end{equation}

\end{thm}

Then Equation~(\ref{HS}) becomes

$$\displaystyle\frac{\partial^2}{\partial \tau^2}\left(\left(1+\frac{k\tau}{2}\right)\varphi\right)=2.$$

Set $\varphi(0)=\varphi'(0)=0$, we obtain the solution 

$$\varphi(\tau)=\displaystyle\frac{\tau^2}{1+\frac{k}{2}\tau}.$$

It's clear from the expression that $\varphi(\tau)>0$ on $I=(0, \infty)$. From the proof of Theorem~\ref{Thm C}, we know for our case, $c_0>0$ and Theorem~\ref{Thm C} ensures the existence of a complete SFK metric $g_0$ on the punctured disc bundle. In fact, it can be further extended to the punctured line bundle in our special situation. 

\begin{lem}
    The metric $g_0$ on $\Delta^*(L)$ can be extended to a complete SFK metric on $L^*$, and near $\tau=0$, it's a finite-area cusp.
\end{lem}

\begin{proof}
    Given the momentum profile $\varphi$, note the log of the fiberwise norm function has the expression  $r=\displaystyle\int_{\tau_0}^{\tau}\displaystyle\frac{dx}{\varphi(x)}$. We define $s(r)\coloneqq\displaystyle\int_{\tau_0}^{\tau}\displaystyle\frac{dx}{\displaystyle\sqrt{\varphi(x)}}.$ Then as $x\rightarrow\infty$, $\varphi(x)$ grows linearly, suggesting that the $r$ integral is divergent, hence the metric lives on the punctured line bundle $L^*$. The $s$ integral diverges, which means the distance to the level set $\{r=r_{\infty}\}$ is unbounded; hence the metric is complete. Here $r_{\infty}=\displaystyle\lim_{\tau\rightarrow\infty}\displaystyle\int_{\tau_0}^{\tau}\displaystyle\frac{dx}{\varphi(x)}.$ 
    
   Near $\tau=0$, the momentum profile has a double zero,
   
$$\varphi(\tau)=a\tau^2+O(\tau^3),\qquad a>0.$$

By the momentum-profile criterion \cite[Proposition~2.3, Table~2.1]{HS02}, this corresponds to a finite-area cusp: the distance $\int_0^\epsilon \frac{d\tau}{\sqrt{\varphi(\tau)}}$ is infinite, while the fibre area is finite. Hence the metric is complete and has a finite-area cusp near $\tau=0$.
\end{proof}

Note that locally, we can rewrite the $S^1$-invariant scalar-flat K\"ahler metric $\omega_{HS}$ under the \textbf{\textit{LeBrun ansatz}} in \cite{lebrun1991explicit}. On a neighborhood $U\subset\R^3$ with coordinates $x,y,\tau$, define $w:U\rightarrow\R$ to be the inverse of $\varphi(\tau)$, and 
we define a function $u: U\rightarrow\R$ by

$$u\coloneqq\log \displaystyle\frac{\tau^2}{\left(1+\displaystyle\frac{x^2+y^2}{2}\right)^2}.$$

Then $u$ is a solution to the $SU(\infty)$ Toda-Lattice equation 

\begin{equation} \label{lebrun ansatz u}
\partial^2_xu+\partial^2_yu+\partial^2_{\tau}(e^u)=0,
\end{equation}

and $u, w$ satisfies

\begin{equation} \label{lebrun ansatz w}
 \partial^2_xw+\partial^2_yw+\partial^2_{\tau}(we^u)=0.
\end{equation}

Hence locally, for some connection $1$-form $\theta$, we can write

\begin{equation}\label{LeBrun ansatz}
    g_{HS}=e^uw(dx^2+dy^2)+wd\tau^2+w^{-1}\theta^2.
\end{equation}

As discussed in \cite{HS02}, when $\varphi(0)=0, \varphi'(0)=2$, we obtain the classical Burns--Simanca metric. In fact, assuming $\varphi(0)=0, \varphi'(0)=2\beta$ for $\beta\in(0, 1)$, we get

$$\varphi(\tau)=\displaystyle\frac{\tau(\tau+2\beta)}{1+\frac{k}{2}\tau},$$

and hence, obtain a family of conical scalar-flat K\"ahler metrics with cone angle $\beta$ connecting the classical Burns--Simanca and cuspidal metrics. Locally, for $\beta\in[0, 1]$, we can write down this family of $S^1$-invariant metrics in the LeBrun ansatz as in \ref{LeBrun ansatz}, with

$$w=\displaystyle\frac{1+\frac{k}{2}\tau}{\tau(\tau+2\beta)}, \quad u=\log\displaystyle\frac{\tau(\tau+2\beta)}{\left(1+\displaystyle\frac{x^2+y^2}{4}\right)^2}$$

satisfying Equations \eqref{lebrun ansatz u} and \eqref{lebrun ansatz w}.

\section{Further study on the cuspidal scalar-flat K\"ahler metric}\label{section 3}


In this section, we study the behavior of the metric near both the asymptotically Euclidean(AE) end and the cusp end. We start with the simplest setting, where the metric lives on the total space of the punctured line bundle $\mathcal{O}(-1)^*$ over $\CP1$, then extend the arguments to the general setting of total space of $\mathcal{O}(-k)^*$ over $\CP {n-1}$. Recall on $\mathcal{O}(-1)^*$,
$$\varphi(\tau)=\displaystyle\frac{\tau^2}{1+\frac{\tau}{2}},$$
from the relations given by the momentum construction, $\frac{dr}{d\tau}=\varphi(\tau)^{-1}, \frac{df(r)}{d\tau}=\frac{\tau}{\varphi(\tau)}$, we get

\begin{equation}\label{t in terms of tau}
    r=\displaystyle\int\displaystyle\frac{1+\frac{x}{2}}{x^2}\,dx=\frac{1}{2}\log\tau-\frac{1}{\tau},
\end{equation}

$$f(r)=\displaystyle\int\displaystyle\frac{1+\frac{x}{2}}{x}\,dx=\log\tau+\frac{\tau}{2}.$$

Then from (\ref{t in terms of tau}), we can express $\tau$ in terms of $|z|_h^2$ 

$$\tau=\displaystyle\frac{2}{W\left(\displaystyle\frac{2}{|z|_h^2}\right)},$$

where $W$ is the product log function. A key property of the product log function we are going to use is that as $x\rightarrow\infty$,

$$W(x)=\log x-\log\log x+o(1).$$

\paragraph{Behavior near $E$}

As $|z|_h\rightarrow 0$, the dominating term for $W(x)$ is $\log x$, more precisely, for a given $\delta>0$ sufficiently small, there exists $\epsilon$, such that when $0<|z|_h<\epsilon$, we have

$$W\left(\displaystyle\frac{2}{|z|_h^2}\right)=(1-\delta)\log\left(\displaystyle\frac{2}{|z|_h^2}\right)+o(1), $$

and $f(t)=\log \displaystyle\frac{2}{1-\delta}-\log(\log 2-\log|z|_h^2)+\displaystyle\frac{1}{(1-\delta)(\log 2-\log|z|_h^2)}$. Then

$$\omega_{HS}=2p^*\omega_{FS}-2\sqrt{-1}\partial\bar{\partial}\left(\log(\log 2-\log|z|_h^2)+\displaystyle\frac{1}{(1-\delta)(\log 2-\log|z|_h^2)}\right).$$

Furthermore, we can control the growth of the norm of the derivatives of $\omega_{HS}$. Note

$$W'(x)=\displaystyle\frac{W(x)}{x(W(x)+1)},$$

when $x$ is sufficiently large, we get $\displaystyle\frac{W(x)}{W(x)+1}=1+O\left(\frac{1}{\log x}\right)$ , thus we write in this case

$$W'(x)=\displaystyle\frac{1}{x}\left(1+O\left(\frac{1}{\log x}\right)\right).$$

In general, the $n$th derivative of the product log function is given by 

$$W^{(n)}(x)=\displaystyle\frac{W^{n-1}(x)}{x^n\left(1+W(x)\right)^{2n-1}}\displaystyle\sum_{k=1}^na_{k, n}W^k(x).$$

When $x$ is sufficiently large, the dominating term for $W^{(n)}(x)$ is $a_{n, n}\displaystyle\frac{W^{2n-1}(x)}{z^n\left(1+W(x)\right)^{2n-1}}$. Note the term $a_{n, n}$ satisfies the recurrence relation $a_{n, n}=-(n-1)a_{n-1, n-1}$ with $a_{1,1}=1$. Thus 

$$a_{n,n}=(-1)^{n-1}(n-1)!.$$

Hence, when $x$ is sufficiently large, we can write

$$W^{(n)}(x)=\displaystyle\frac{(-1)^{n-1}(n-1)!}{x^n}\left(1+O\left(\frac{1}{\log x}\right)\right).$$

At the same time, we know $$\displaystyle\frac{d^n\log x}{dx^n}=\displaystyle\frac{(-1)^{n-1}(n-1)!}{x^n},$$

then for index $I$ with $|I|=k$, we know

$$\displaystyle\frac{\partial^k\tau}{\partial z^I}=\displaystyle\frac{d^k \left(\displaystyle\frac{2}{\log\frac{2}{|z|_h^2}}\right)}{dz^I}\left(1+O\left(\frac{1}{\log|z|_h^2}\right)\right).$$

The goal is to show that the Hwang-Singer metric is of Poincar\'e type. For this, we consider the reference metric $\omega_h\coloneqq 2p^*\omega_{FS}-\sqrt{-1}\partial\bar{\partial}\log(\log2-\log|z|_h^2)$, and let $f_h\coloneqq -\log(\log2-\log|z|_h^2)$. From \cite{Auv17}, we know that if we take a polydisc $U$ around $E$ and write $E\cap U=\{z_1=0\}$, $U\backslash E=\Delta^*\times\Delta$, then near $E$, $\omega_h$ is mutually bounded by a product of Poincar\'e metric on punctured disc with a smooth metric $\omega_{U}=\displaystyle\frac{\sqrt{-1}dz_1\wedge d\bar{z}_1}{|z_1|^2\log^2|z_1|^2}+\omega_h|_E,$ which has bounded derivatives at any order with respect to an orthonormal frame to be defined below. Write $\tau_h\coloneqq \displaystyle\frac{2}{\log\frac{2}{|z|_h^2}}$, then $\tau=\displaystyle\frac{1}{1-\delta}\tau_h$, and for $k\ge 2$,

$$\displaystyle\frac{\partial^k f}{\partial^k \tau}=\displaystyle\frac{\partial^{k-1}\left(\frac{1}{\tau}+\frac{1}{2}\right)}{\partial^{k-1}\tau}=\displaystyle\frac{\partial^{k-2}\left(\frac{1}{-\tau^2}\right)}{\partial^{k-2}\tau}, \quad \displaystyle\frac{\partial^k f_h}{\partial^k \tau_h}=\displaystyle\frac{\partial^{k-1}\left(\frac{1}{\tau_h}\right)}{\partial^{k-1}\tau_h}=\displaystyle\frac{\partial^{k-2}\left(\frac{1}{-\tau_h^2}\right)}{\partial^{k-2}\tau_h}.$$

It follows that $\displaystyle\frac{\partial^k f}{\partial^k \tau}=\displaystyle\frac{1}{(1-\delta)^k}\displaystyle\frac{\partial^k f_h}{\partial^k \tau_h}$.

We fix a coframe $e^1=dz^1, e^2=\displaystyle\frac{1}{|z|_h\log|z|_h}dz^2$, then for $k\ge 2$,

\begin{equation*}
    \displaystyle\frac{\partial^k f}{\partial e^I}=\displaystyle\frac{\partial^k f}{\partial^k \tau}\displaystyle\frac{\partial^k\tau}{\partial z^I}\displaystyle\frac{\partial z^I}{\partial e^I}de^I=\displaystyle\frac{1}{(1-\delta)^k} \displaystyle\frac{\partial^k f_h}{\partial^k \tau_h}\displaystyle\frac{\partial\tau^k}{\partial z^I}\displaystyle\frac{\partial z^I}{\partial e^I}de^I
    =\displaystyle\frac{1+O\left(\frac{1}{\log|z|_h^2}\right)}{(1-\delta)^{k}}\displaystyle\frac{\partial^k f_h}{\partial^k \tau_h}\displaystyle\frac{d^k\tau_h}{dz^I}\displaystyle\frac{\partial z^I}{\partial e^I}de^I=\displaystyle\frac{1+O\left(\frac{1}{\log|z|_h^2}\right)}{(1-\delta)^k}\frac{\partial^kf_h}{\partial e^I},
\end{equation*}

hence for $j\ge 2$, given $\dt$ and $|z|_h$ sufficiently small,

$$\nabla^j_{\omega_U}f=\nabla^j_{\omega_U}f_h\cdot\left(1+O\left(\frac{1}{\log|z|_h^2}\right)\right).$$

Or equivalently, $\nabla^j_{\omega_h}f=\nabla^j_{\omega_h}f_h\cdot\left(1+O\left(\frac{1}{\log|z|_h^2}\right)\right)$ for $j\ge 2$. Similar arguments show that $\frac{1}{2}\omega_h\le\omega_{HS}\le 2\omega_h$ and $\nabla_{\omega_h}f=\nabla_{\omega_h}f_h\cdot\left(1+O\left(\frac{1}{\log|z|_h^2}\right)\right)$. From these discussions, we conclude that $\omega_{HS}$ is a Poincar\'e type K\"ahler metric.


\begin{rmk}
In fact, from our computation, we further see the Hwang-Singer metric is actually asymptotically hyperbolic, which is a special type of Poincar\'e type K\"ahler metric. We also mention that unlike here, where we used explicit computation to show the metric is asymptotically hyperbolic, Sz\'ekelyhidi provided a different proof; for details, see \cite{Sze06} Theorem 5.1.3.  
\end{rmk}

\paragraph{Asymptotic behavior away from $E$}\label{asymptotic away from E}
As $|z|_h\rightarrow\infty$, $\log\tau=2\log|z|_h+O(\tau^{-1})$, thus 

$$f=\log|z|_h^2+\frac{1}{2}|z|_h^2+1+O(|z|_h^{-2}),$$ and $\omega_{HS}=\sqrt{-1}\partial\bar{\partial}\left(|z|_h^2+4\log|z|^2_h+O(|z|_h^{-2})\right)$. We rescale $\omega_{HS}$ by a factor of $\frac{1}{2}$ so that it lives in the K\"ahler class $[\omega_{FS}]$ and perform change of variables $z\mapsto\frac{z}{\sqrt 2}$, then near the divisor $E$, 

\begin{equation}\label{asymptotics of HS n=2}
\omega_{HS}=\displaystyle\sqrt{-1}\partial\bar{\partial}\left(|z|_h^2+2\log |z|_h+\varphi_2(z)\right),
\end{equation}

 for some $\varphi_2=O(|z|_h^{-2})$; near $E$, 

 $$\omega_{HS}=\ddb\left(\log|z|_h^2-\log(\log2-\log|z|_h^2)+O\left(\frac{1}{\log|z|_h^2}\right)\right).$$

\paragraph{$\mathcal{O}(-k)$ case}

We have studied the cuspidal metric on $\mathcal{O}(-1)^*$; in fact, similar behavior of the metric can be seen in $\mathcal{O}(-k)^*$ for $k>1$. Note

\begin{equation}\label{t in terms of tau for k}
    r=\displaystyle\int\displaystyle\frac{1+\frac{kx}{2}}{x^2}\,dx=\frac{k}{2}\log\tau-\frac{1}{\tau},
\end{equation}

$$f(r)=\displaystyle\int\displaystyle\frac{1+\frac{kx}{2}}{x}\,dx=\log\tau+\frac{k\tau}{2}.$$

Near $\tau=0$, we solve for $\tau$, and get

$$\tau=\displaystyle\frac{\displaystyle\frac{2}{k}}{W\left(\displaystyle\frac{2}{k}\left(\displaystyle\frac{1}{|z|^2_h}\right)^{1/k}\right)}.$$

A similar analysis shows that it is still of Poincar\'e type as $\tau\rightarrow 0$. For asymptotic behavior away from $E$, let $|z|_h^2\rightarrow\infty$, we have $\log\tau=\frac{2}{k}\log|z|_h+O(\tau^{-1})$. Consider 

$$q:\mathcal{O}(-k)\rightarrow\mathbb{C}^{n}\slash\Gamma_k,$$

precomposing with it we obtain $f(r\circ q^{-1})=(\log\tau+k\tau/2)\circ q^{-1}$, and the expression for the potential is written as

$$\displaystyle\frac{k}{2}\sum|z_i|^2+\log\sum|z_i|^2.$$

Similar to above, after suitable rescaling of the metric and change of variables, we can normalize $\omega_{HS}$ such that near the AE end

$$\omega_{HS, k}=\displaystyle\sqrt{-1}\partial\bar{\partial}\left(|z|_h^2+c\log |z|_h+\varphi_{2,k}(z)\right),$$ 

for some $\varphi_{2,k}=O(|z|_h^{-2})$ and some constant $c$.

\paragraph{Higher dimension case}
The previous discussion are restricted to where the base manifold is $\C\mathbb{P}^{1}$. The next step is to discuss the situation of higher dimension where the base manifold is $\C\mathbb{P}^{n-1}$ for $n\ge 3$. On $\mathcal{O}(-k)$ and $\beta=-\frac{k}{2}$, take $\omega_M=\omega_{FS}$ where Ric $\omega_{FS}=n\cdot\omega_{FS}$, then solving the SFK equation we obtain

\begin{equation}\label{higher dimension expression}
    \varphi(\tau)=\displaystyle\frac{2}{n\beta^2}\left(1-\beta\tau+\displaystyle\frac{n-1}{(1-\beta\tau)^{n-1}}-\displaystyle\frac{n}{(1-\beta\tau)^{n-2}}\right).
\end{equation}

For notational simplicity, we write the computation in the case $\beta=-1$, equivalently $k=2$ under the normalization $\beta=-k/2$. The general case differs only by a constant rescaling of the momentum variable. For general $k$, let

$$a=-\beta=\frac{k}{2},\qquad s=a\tau.$$

Then the expansion of the momentum profile is reduced to the same form as in the case $a=1$, with constants depending on $k$. The powers appearing in the asymptotic expansions are unchanged.

As $\tau\rightarrow\infty$, we have

$$r=\frac{1}{2}\displaystyle\int\displaystyle\frac{1}{1+x}\left(1+\displaystyle\frac{1+nx}{(1+x)^n}+O(x^{2-2n})\right)\,dx=\frac{1}{2}\log(1+\tau)+d_1(1+\tau)^{1-n}+d_2(1+\tau)^{-n}+O(\tau^{2-2n}),$$

$$f(r)=\frac{1}{2}\displaystyle\int\left(1-\displaystyle\frac{1}{1+x}\right)\left(1+\displaystyle\frac{1+nx}{(1+x)^n}+O(x^{2-2n})\right)\,dx=\frac{1}{2}\tau-\frac{1}{2}\log(1+\tau)+d_3(1+\tau)^{1-n}+d_4(1+\tau)^{-n}+O(\tau^{3-2n}),$$

where $d_1, \cdots, d_4$ are all constants. It follows that

$$f=\frac{1}{2}|z|^2_h-\frac{1}{2}\log|z|_h+d_3|z|_h^{4-2n}+|z|^{2-2n}+O(|z|_h^{6-4n}),$$

and the metric is written as

\begin{equation}\label{asymptotics of HS n ge 3}   \omega_{HS}=\displaystyle\sqrt{-1}\partial\bar{\partial}(|z|_h^2+\varphi_2(z)),
\end{equation}

for some $\varphi_2=d_3|z|_h^{4-2n}+O(|z|_h^{2-2n})$.

\medskip

For small $\tau$, as $\tau\rightarrow 0$, write $r=a_{-1}\frac{1}{\tau}+a_0\log\tau+O(\tau)$, then $\log|z|_h^2=\frac{a_{-1}}{\tau}+a_0\log\tau+O(\tau)$, giving us

$$\tau=\displaystyle\frac{-a_{-1}}{a_0W\left(\displaystyle\frac{-a_{-1}}{a_0}\left(\displaystyle\frac{1}{|z|^2_h}\right)^{2/a_0}\right)},$$

showing that they are still of Poincar\'e type. Tracing the coefficients more carefully, we obtain the following expansion

$$\omega_{HS}=\ddb\left(\log|z|_h^2-\frac{2}{n(n-1)}\log(a-\log|z|^2_h)+O\left(\frac{1}{\log|z|_h^2}\right)\right),$$

for some constant $a$. This expansion is also discussed in \cite{FYZ16}. 

For the remaining cases where $k\ne 2$, the situation for $\mathcal{O}(-k)$ is similar, as $|z|_h^2\rightarrow\infty$, 

$$\omega_{HS, k}=\displaystyle\sqrt{-1}\partial\bar{\partial}(|z|_h^2+\varphi_{2, k}(z)),$$

for some $\varphi_{2 ,k}=O(|z|_h^{4-2n})$. As $|z|_h^2\rightarrow 0$, it is again of Poincar\'e type.

\section{Set up of the construction}\label{section 4}

With the scalar-flat K\"ahler Poincar\'e type metric $\omega_{HS}$, we are ready to discuss the cscK equation we aim to solve for the gluing construction. Given $p\in X$, consider a small neighborhood of $p\in X$, with $z$ being the complex normal coordinate, WLOG we assume $|z|<1$. Now we rewrite $\omega_X$ as

$$\omega_X=\displaystyle\sqrt{-1}\partial\bar{\partial}\left(|z|^2+\varphi_1(z)\right),$$

for some $\varphi_1=O(|z|^4)$. For a small constant $\ep>0$, we consider the radius $0<r_\epsilon<1$ of the form 

\begin{equation}\label{choice of r}
r_{\epsilon}=\epsilon^{\frac{2n-1}{2n+1}}.
\end{equation}

This particular exponent is chosen for later estimates to be discussed in Section~\ref{section 8}. We aim at gluing the metric

$$\omega_{HS}=\sqrt{-1}\partial\bar{\partial}
\begin{cases}
|\zeta|^2+\log |\zeta|+\varphi_2(\zeta) & \text{for } n=2,\text{ with } \varphi_2=O(|\zeta|^{-2}),\\
|\zeta|^2+\varphi_2(\zeta) & \text{for } n\ge 3,\text{ with } \varphi_2=O(|\zeta|^{-2n+4})
\end{cases}$$

on $Bl_0\C^n\backslash E$, where $\zeta$ denotes the coordinate on the Hwang--Singer model, to $\omega_X$ on an annulus region $A_{r_\epsilon}\coloneqq B_{2r_\epsilon}\backslash B_{r_\epsilon}$. For this, let $z$ denote the coordinate on $A_{r_\epsilon}$ and set

$$\zeta=\epsilon^{-1}z.$$

Equivalently, setting $\zeta=F_\epsilon(z)=\epsilon^{-1}z$, we have
$$\epsilon^2F_\epsilon^*\omega_{HS}=\sqrt{-1}\partial\bar{\partial}
\begin{cases}
|z|^2+\epsilon^2\log|\zeta|+\epsilon^2\varphi_2(\zeta), & n=2,\\
|z|^2+\epsilon^2\varphi_2(\zeta), & n\ge 3.
\end{cases}$$
Thus the gluing construction matches $\epsilon^2F_\epsilon^*\omega_{HS}$ with $\omega_X$ on the annulus $A_{r_\epsilon}\coloneqq B_{2r_\epsilon}\backslash B_{r_\epsilon}$.

Consider a cut-off function $\gamma: \mathbb{R}\rightarrow [0,1]$ satisfying

$$\gamma(x) = \begin{cases}
  0  & x\le 1, \\
  1 & x\ge 2
\end{cases}$$

we rescale it to $\gamma_1(z)=\gamma\left(\displaystyle\frac{|z|}{r_{\epsilon}}\right)$, and let $\gamma_2=1-\gamma_1$. Note although both $\gamma_1$ and $\gamma_2$ depend on $\ep$, we will omit them in the notation for simplification. Now, we define an approximate solution of the cscK equation as our candidate for the actual solution,

\begin{equation}\label{the perturbed metric}
    \omega_{\epsilon} \coloneqq  \begin{cases}
  \omega_X  & \text{on } X\backslash B_{2r_\epsilon},\\
    \sqrt{-1}\partial\bar{\partial}\left(|z|^2+\gamma_1(z)\varphi_1(z)+\epsilon^2\gamma_2(z)\varphi_2(\zeta)\right) & \text{on } A_{r_\epsilon} \text{ for } n\ge 3,\\

   \sqrt{-1}\partial\bar{\partial}\left(|z|^2+\gamma_1(z)\varphi_1(z)+\epsilon^2\gamma_2(z)(\log|\zeta|+\varphi_2(\zeta))\right) & \text{on } A_{r_\epsilon} \text{ for } n=2,\\

  \epsilon^2\omega_{HS} & \text{on } B_{r_\epsilon}\backslash p.
\end{cases}
\end{equation}

First note given $\epsilon$ sufficiently small, the metric $\omega_{\epsilon}$ is positive definite everywhere. It's because under direct calculation, we see on $A_{r_{\epsilon}}$, $\gamma_1(z)\varphi_1(z)+\epsilon^2\gamma_2(z)(\varphi_2(\epsilon^{-1}z)+\log|\epsilon^{-1}z|)=O(|z|^4)$ for $n=2$ and $\gamma_1(z)\varphi_1(z)+\epsilon^2\gamma_2(z)(\varphi_2(\epsilon^{-1}z))=O(|z|^4)$ for $n\ge 3$. Then by considering the preimage of $B_{r_{\epsilon}}$ under the blow-down map $\pi$, we can view $\omega_{\epsilon}$ as living on $Bl_pX\backslash E$.

\begin{rmk}
    For a geometric meaning of $\ep$, $\ep^{2n-2}$ can be viewed as the volume of $E$ for our gluing construction.
\end{rmk}

We decompose the operator into linear and non-linear parts to approach the metric defined by this non-linear PDE of the constant scalar curvature equation. Let

$$L_{\omega}(\varphi)\coloneqq \displaystyle\frac{d}{dt}\Bigg|_{t=0}S(\omega+t\sqrt{-1}\partial\bar{\partial}\varphi)=-\mathcal{D}^*_{\omega}\mathcal{D}_{\omega}\varphi+g^{j\bar{k}}\partial_jS(\omega)\partial_{\bar{k}}\varphi$$

be the linearization of the scalar curvature operator. Here $\mathcal{D}_{\omega}\varphi=\bar{\partial}\circ\nabla^{1,0}_{\omega}\varphi$, and 

$$\mathcal{D}^*_{\omega}\mathcal{D}_{\omega}\varphi=\frac{1}{2}\Delta^2_{\omega}\varphi+\langle i\partial\bar{\partial}\varphi, \rho\rangle+\frac{1}{2}\langle d\varphi, dS(\omega)\rangle,$$

is known as the Lichnerowicz operator. $\rho$ denotes the Ricci form for $\omega$. When $\omega$ is cscK, the linearization $L_{\omega}(\varphi)$ coincides with the negative of the Lichnerowicz operator. 

\medskip

Consider the following decomposition of $Bl_pX\backslash E$. Set

$$R_{\epsilon}\coloneqq \displaystyle\frac{r_{\epsilon}}{\epsilon}, \quad X_{r_{\ep}}\coloneqq X\backslash B_{r_{\ep}},$$

then 

$$Bl_pX\backslash E=X_{r_{\ep}}\coprod Bl_pB_{R_{\ep}}\backslash E.$$

We study the Lichnerowicz operator in order to construct, on each piece, a family of cscK metrics parametrized by suitable boundary data. The gluing problem is then reduced to choosing the boundary data so that the two families match over the gluing annulus. 

We follow the framework of Arezzo--Pacard's blow-up construction, as adapted to the Poincar\'e type setting in \cite{Sek18}. This is well suited to the present problem, since the construction naturally splits into $X_{r_\epsilon}$ and the Hwang--Singer model piece $Bl_0B_{R_\epsilon}\backslash E$. A major difference in our case is that, due to the complexity introduced by the behavior of $\omega_{HS}$, we need to first consider solving a modified cscK equation on $Bl_0\C^n\backslash E$, to be introduced in Section \ref{section 8}. It would also be very interesting to adapt Sz\'ekelyhidi's compact blow-up framework to the present Poincar\'e type setting.

For later purposes, we will modify $\omega_{\ep}$ a bit on the gluing annulus to obtain a better approximation. In the annulus region, the metric is close to the Euclidean metric, and thus, the Lichnerowicz operator is close to the standard bilaplacian operator. We want to find biharmonic functions on each piece with given boundary data.

\begin{lem}(\cite{AP09}, Proposition 4.3, 4.4)
    Given $h\in C^{4, \alpha}(\partial B_1), k\in C^{2, \al}(\partial B_1)$, there exists $H^1_{h, k}\in C^{4, \al}(\bar{B}_1)$ satisfying

    $$\Delta^2_{Euc}H^1_{h, k}=0 \text{ in }B_1;\quad H^1_{h, k}=h,\quad \Delta_{Euc} H^1_{h, k}=k \text{ on }\partial B_1,$$

    with the following control of bounds

    $$\|H^1_{h, k}\|_{C^{4, \al}(\bar{B}_1)}\lesssim \|h\|_{C^{4, \alpha}(\partial B_1)}+\|k\|_{C^{2, \al}(\partial B_1)}.$$

    Given $h\in C^{4, \alpha}(\partial B_1), k\in C^{2, \al}(\partial B_1)$, and further $\displaystyle\int_{\partial B_1}k=0$, there exists  $H^0_{h, k}\in C^{4, \alpha}_{4-2n}(\C^n\backslash B_1)$ satisfying
    
    $$\Delta^2_{Euc}H^0_{h, k}=0 \text{ in }\C^n\backslash \bar{B}_1;\quad H^0_{h, k}=h, \quad \Delta_{Euc} H^0_{h, k}=k \text{ on }\partial B_1,$$

    with the following control of bounds

    $$\|H^0_{h, k}\|_{C^{4, \al}_{3-2n}(\C^n\backslash B_1)}\lesssim \|h\|_{C^{4, \alpha}(\partial B_1)}+\|k\|_{C^{2, \al}(\partial B_1)}.$$
\end{lem}

Here $H^{0}_{h,k}$ lives in the weighted H\"older space and we postpone its definition to the next section. $C^{k,\alpha}(\overline B_1)$ denotes the standard Euclidean H\"older space up to the boundary. More precisely, identifying $\mathbb C^n$ with $\mathbb R^{2n}$, we set
$$C^{k,\alpha}(\overline B_1)=\left\{u\in C^k(B_1):D^\beta u \text{ extends continuously to } \overline B_1 \text{ for } |\beta|\le k,\ \|u\|_{C^{k,\alpha}(\overline B_1)}<\infty\right\}.$$


With this Lemma, for radius $R_{\ep}$, we construct perturbations on each of the two pieces $Bl_0 B_{R_{\ep}}\backslash E$ and $X_{{r_{\epsilon}}}$. On $Bl_0 B_{R_{\ep}}\backslash E$, given boundary data $\tilde{h}, \tilde{k}$, for cut-off function $\chi^1$ which vanishes on $Bl_0B_1\backslash E$ and equal to $1$ on $Bl_0\C^n\backslash Bl_0B_2$, we define on $Bl_0 B_{R_{\ep}}\backslash E$ the perturbed function

\begin{equation}\label{choice of H^1}
    H^1_{\epsilon, \tilde{h},\tilde{k}}(z)\coloneqq \chi^1(z)\left(H^1_{\tilde{h}^{\perp},\tilde{k}}\left(\displaystyle\frac{z}{R_{\epsilon}}\right)-H^1_{\tilde{h}, \tilde{k}}(0)\right)+H^1_{\tilde{h}, \tilde{k}}(0).
\end{equation}

On $X_{{r_{\epsilon}}}$, given boundary data $h , k$, for a cut-off function $\chi^0$ which vanishes on $X\backslash B_2$ and equals to $1$ on $B_1\backslash{p}$, we define on $X_{r_{\ep}}$

$$H^0_{\epsilon, h, k}(z)\coloneqq \chi^0(z)H^0_{h,k}\left(\displaystyle\frac{z}{r_{\epsilon}}\right).$$

On $X_{r_{\epsilon}}$, we follow \cite{AP09} and perturb the metric with smooth function $\varphi_{\epsilon, h,k}$ such that

\begin{equation}\label{metric on base piece}
    \omega_{\epsilon, h,k}\coloneqq\omega_X+\displaystyle\sqrt{-1}\partial\bar{\partial}(H^0_{\epsilon, h,k}+\varphi_{\epsilon, h,k})
\end{equation}

satisfying cscK equation, for a more precise statement, see Theorem~\ref{theorem cscK on punctured base}.

\medskip

On $Bl_0 B_{R_{\ep}}\backslash E$, taking the perturbed biharmonic extension into account, we want to find a smooth potential function $u_{\epsilon, \tilde{h},\tilde{k}}$ such that the metric

\begin{equation}\label{metric on HS piece}
    \omega_{\epsilon, \tilde{h}, \tilde{k}, HS}\coloneqq\omega_{HS}+\displaystyle\sqrt{-1}\partial\bar{\partial}(H^1_{\epsilon, \tilde{h},\tilde{k}}+u_{\epsilon, \tilde{h},\tilde{k}})
\end{equation}

satisfies a modified cscK equation, which we will introduce in Section~\ref{section 8}. To solve the modified cscK equation, we will first aim at understanding the linearized operator $L_{\omega_{HS}}$ in later sections.

\begin{rmk}
Note that our discussions here and in the upcoming sections are restricted to the situation of blowing up one point, but there are no essential differences when extending the arguments to the situation of blowing up finitely many points $\{p_1, \cdots, p_k\}$. With sufficiently small $\ep$, let

$$X^*_{r_{\ep}}\coloneqq X-\{B_{r_{\ep}}(p_1), \cdots, B_{r_{\ep}}(p_k)\}, \quad B_{R_{\ep}}^i\coloneqq B_{R_{\ep}}(p_i)$$

for $i=1, \cdots k$. On $X^*_{r_{\ep}}$, standard arguments from \cite{arezzo2006blowing} show the existence of cscK metric on it. On each $Bl_0 B_{R_{\ep}}^i\backslash E_i$, we solve the modified cscK equation as in the case of blowing up one point. The matching argument for finitely many gluing annuli is essentially the same for matching over one annulus. Hence, understanding the situation of blowing up one point of $p$ suffices.
\end{rmk}

\section{Weighted H\"older space}\label{section 5}

In this section, we introduce doubly weighted H\"older spaces on $\bl_0\C^n\setminus E$. This is a modification of the weighted spaces introduced by Auvray\cite{Auv14} and used by Sektnan\cite{Sek18}. Auvray's spaces involve only the Poincar\'e weight,  while in Sektnan's blow-up construction the points being blown up lie away from the Poincar\'e divisor. In our setting, the local model is $\bl_0\C^n\setminus E$, where the exceptional divisor itself exhibits the Poincar\'e type behavior, and it also has an AE end. The two weights keep track of the behavior near the Poincar\'e divisor and at the AE end separately.

\medskip

To define H\"older spaces on $X\backslash D$, fix local coordinates $(z_1,\ldots,z_n)$ near $D$ such that $D=\{z_1=0\}$. Since $D$ is smooth, there is only one punctured coordinate. For $\xi\in(0,1)$, consider

$$\Phi_{\xi}: \frac{3}{4}\Delta\times\Delta^{n-1}\rightarrow \Delta^*\times\Delta^{n-1},$$

given by $(w,z_2,\ldots,z_n)\mapsto(\varphi_{\xi}(w),z_2,\ldots,z_n)$. Here $\varphi_\xi:\frac{3}{4}\Delta\to \Delta^*$ is the usual quasi-coordinate map on the punctured disc. 

Let $\{U_i\}_{i=1}^{\ell}$ be a finite set of charts such that $\left(\bigcup_{i=1}^{\ell}U_i\right)\cap D$ covers $D$ and let $U_0$ be an open set satisfying $X=\bigcup_{i=0}^{\ell} U_i$. For a function on $X\backslash D$, we define

$$\|f\|_{C^{k, \alpha}(X\backslash D)}\coloneqq \|f\|_{C^{k, \alpha}(U_0)}+\max_{i=1}^{\ell}\|f\|_{C^{k, \alpha}(U_i\backslash D)}.$$

The next goal is to add \textit{weight} to this H\"older space. Note for Poincar\'e metric $\omega_{\Delta^*}$, we can change to the \textit{logarithmic polar coordinates} $(t, \theta)\in\mathbb{R}\times\mathbb{S}^1$, where $\theta$ is the usual angular coordinate, and

$$t\coloneqq \log\left(-\log(|z|^2)\right).$$

For $\omega_{PT}$ on $X\backslash D$, on a neighborhood of $D$, we can similarly define a function $t$ which is invariant under the $S^1$-action and  differs from $h=\log(\lambda-\log|\sigma|^2)$ by an $O(e^{-h})$ term. Then, the weighted H\"older norm is defined as 

\begin{equation}\label{def weighted norm}
    \|f\|_{C^{k, \alpha}_{\eta}(X\backslash D)}\coloneqq \|e^{t\eta}f\|_{C^{k, \alpha}(X\backslash D)}.
\end{equation}

\begin{rmk}\label{cylinder norm}
    As pointed out in \cite{Auv13-1} Section 3 and \cite{Sek16} Chapter 6, if we consider a neighborhood of $D$ in the holomorphic normal bundle, it can be identified with a tubular neighborhood $N(D)$ of $D$ in $X$ via the exp map defined by $\omega_X$. Note $N(D)$ admits an $S^1$-action coming from that in the holomorphic normal bundle, and it gives us an $S^1$-invariant projection $\pi: N(D)\backslash D\rightarrow D.$ Write $N(D)_a\coloneqq \{x\in N(D)\backslash D: t(x)\ge a\},$ we get an $S^1$-fibration 

    $$\Pi=(\pi, t): N(D)_a\backslash D\rightarrow D\times[a, \infty).$$

    Further, we can define a connection $1$-form $\Theta$ as follows. Let $\xi$ be the infinitesimal generator of the $S^1$-action with $\varphi_s$ being the associated flow. Then we define 

    $$\hat{\Theta}_x\coloneqq \displaystyle\int_{0}^{2\pi}\varphi_s^*\left(\displaystyle\frac{g_x(\cdot,\xi)}{g_x(\xi, \xi)}\right)ds,\quad \Theta_x\coloneqq 2\pi\displaystyle\frac{\hat{\Theta}_x}{\displaystyle\int_{S^1}\hat{\Theta}_x}.$$

    Note for  $\theta\coloneqq\displaystyle\sqrt{-1}\left(\log z-\displaystyle\frac{1}{2}e^t\right)$, we have $\|\nabla^j(\Theta-d\theta)\|=O(1)$, we can write

    \begin{equation}\label{expansion of metric with model}
        g=dt^2+e^{-2t}\Theta^2+\pi^*g_D+O(e^{-t}).
    \end{equation}

    For $f\in C^{k, \alpha}_{\eta}(X\backslash D)$, we can write an orthogonal decomposition of $f$ as  $f=f^0+f^{\ne 0},$ where $f^{0}$ is an $S^1$-invariant function on $N(D)$ and $f^{\ne 0}$ has zero mean on each fiber of $\Pi$. We see $f^0$ can be viewed as a function $[a, \infty)\times D\rightarrow\mathbb{R}$ and our norm defined near the Poincar\'e type cusp is \textbf{\textit{equivalent to that for the cylinder}}. More explicitly, the $C^{k, \alpha}$ norm on the cylinder $[a, \infty)\times D$ defined as $\sup_{s\ge a+1}\|f^0\|_{C^{k, \alpha}([s-1, s+1]\times D, dt^2+g_D)}$ is equivalent to the Poincar\'e type norm for $S^1$-invariant function. For details, see \cite{Sek16}, Lemma~6.7. We immediately see that the weighted norms for the cylinder can be defined similarly and are equivalent to those for Poincar\'e type cusp.
\end{rmk}

Let $X=Bl_0 B_1\backslash E$, then the above discussion gives us the weighted norm on $X$. For a function $f: \mathbb{R}^{2n}-\{0\}\rightarrow\R$, consider $f_r: B_2\backslash B_1\rightarrow\mathbb{R}$ defined as

\begin{equation}\label{f_r}
    f_r(z)\coloneqq r^{-\delta}f(rz),
\end{equation}

then its weighted norm is defined as

\begin{equation}\label{f_r weight}
    \|f\|_{C^{k, \alpha}_{\delta}(\mathbb{R}^{2n}-\{0\})}\coloneqq \sup_{r>0}\|f_r\|_{C^{k, \alpha}(B_2\backslash B_1)}.
\end{equation}

To define weighted H\"older space on $Bl_0\mathbb{C}^n\backslash E$, we use \textbf{\textit{double index}} $\eta, \delta$ which are independent of each other to monitor the growth rate of $f$ near the divisor and on the AE end separately.

\begin{definition}\label{definition norm}
    For a function $f$ defined on $Bl_0\mathbb{C}^n\backslash E$, the H\"older norm with weight $\eta, \delta$ is defined as

\begin{equation}
    \|f\|_{C^{k, \alpha}_{\eta, \delta}(Bl_0\mathbb{C}^n\backslash E)}\coloneqq \|f\|_{C^{k, \al}_{\eta}(Bl_0B_{2}\backslash E)} +\|\gamma f\|_{C^{k, \alpha}_{\delta}(Bl_0\mathbb{C}^n\backslash B_{1})},
\end{equation}

where $\gamma$ is a cut-off function which equals $1$ outside $Bl_0 B_2\backslash E$ and equals $0$ in $Bl_0B_1\backslash E$, then $\gamma f$ can be viewed as a function on $\mathbb{R}^{2n}-\{0\}$, whose norm is defined as in Equation~\eqref{f_r weight}.
\end{definition}

We also introduce the weighted Sobolev space on $Bl_0\mathbb{C}^n\backslash E$ for later use.

\begin{definition}
For a function $f$ defined on $Bl_0\mathbb{C}^n\backslash E$, the Sobolev norm with weights $\eta,\delta$ is defined as
\begin{equation}
    \|f\|_{W^{k,2}_{\eta,\delta}(Bl_0\mathbb{C}^n\backslash E)}
    \coloneqq
    \sum_{i=0}^k
    \|\nabla^i f\|_{L^2_{\eta,\delta}(Bl_0\mathbb{C}^n\backslash E)}.
\end{equation}
Here the weighted $L^2$ norm is defined by
\begin{equation}
    \|f\|_{L^2_{\eta,\delta}(Bl_0\mathbb{C}^n\backslash E)}
    \coloneqq
    \|\Gamma_{\eta,\delta}f\|_{L^2(Bl_0\mathbb{C}^n\backslash E)}.
\end{equation}
The function $\Gamma_{\eta,\delta}$ is a smooth positive function which coincides with $e^{\eta t}$ near the divisor and with $r^{-\delta}$ near the AE end. Geometrically, the weighted norm measures exponential decay of order $\eta$ along the cusp end and polynomial behavior of order $\delta$ along the AE end.
\end{definition}

An inclusion relation between Sobolev and H\"older weighted spaces is given as follows.

\begin{lem}\label{lemma compare Sobolev and Holder} Given $\eta, \eta', \dt\in\R$, 
    $$C^{k, \al}_{\eta, \dt}(Bl_0\mathbb{C}^n\backslash E)\subset W^{2, k}_{\eta', \dt}(Bl_0\mathbb{C}^n\backslash E) \iff \eta'<\eta+\frac{1}{2}.$$
\end{lem}

\begin{proof}
    Note the weight for the two spaces on the AE end are the same, then the result follows immediately from Lemma 2.3 in \cite{Sek18}.
\end{proof}

\section{Linear theory}\label{section 7}

In this section, we study the mapping properties of the linearized operator $L_{\omega_{HS}}$ on $\bl_0\C^n\backslash E$ between weighted H\"older spaces. The main result is Theorem~\ref{theorem bounded inverse for cusp}, where we prove the surjectivity statement needed for the linearized problem. We begin with the Schauder estimate for $L_{\omega_{HS}}$.

\begin{thm}\label{Schauder estimates}
    Given $f\in L^2_{\eta+\frac{1}{2}, \delta}(Bl_0\mathbb{C}^n\backslash E)$ and suppose $L_{\omega_{HS}}f\in C^{0, \alpha}_{\eta, \delta-4}(Bl_0\mathbb{C}^n\backslash E)$, then $f\in C^{4, \alpha}_{\eta, \delta}(Bl_0\mathbb{C}^n\backslash E)$ and moreover,

    $$\|f\|_{C^{4, \alpha}_{\eta, \delta}(Bl_0\mathbb{C}^n\backslash E)}\lesssim\|f\|_{L^2_{\eta+\frac{1}{2}, \delta}(Bl_0\mathbb{C}^n\backslash E)}+\|L_{\omega_{HS}}f\|_{C^{0, \alpha}_{\eta, \delta-4}(Bl_0\mathbb{C}^n\backslash E)}.$$
\end{thm}

\begin{proof}

Choose $\lambda,\mu$ sufficiently large so that $\{t>\lambda\}$ is contained in the divisor end and $\{r>\mu\}$ is contained in the AE end. Let
$$K_0=\left(\bl_0\C^n\backslash E\right)\setminus\left(\{t>\lambda+2\}\cup\{r>\mu+2\}\right).$$
After enlarging $K_0$ if necessary, we may also assume that it contains the compact set $K_D$ appearing in the zero-average estimate below. On this fixed compact set $K_0$, the weighted norms are equivalent to the unweighted norms. Hence the standard interior elliptic estimate on a slightly larger compact set gives
$$\|f\|_{C^{4,\alpha}_{\eta,\delta}(K_0)}\le C\left(\|f\|_{L^2_{\eta+\frac{1}{2},\delta}(\bl_0\C^n\backslash E)}+\|L_{\omega_{HS}}f\|_{C^{0,\alpha}_{\eta,\delta-4}(\bl_0\C^n\backslash E)}\right).$$

It remains to deal with the two non-compact regions $\{t>\lambda+2\}$ and $\{r>\mu+2\}$. 

For the control near the divisor, decompose $f$ into $f^{0}$ and $f^{\ne0}$ as in Remark~\ref{cylinder norm}. We identify the $S^1$-invariant function $f^0$ with a function on $[\lambda, +\infty)\times E$, then for $s>\lambda+2$, let $Q_i^s=[s-i,s+i]\times E$. The local Schauder estimate gives

$$\|f^0\|_{C^{4,\alpha}(Q_1^s)}\leq C\left(\|L_{\omega_{HS}}f^0\|_{C^{0,\alpha}(Q_2^s)}+\|f^0\|_{L^2(Q_2^s)}\right),$$

where $C$ is independent of $s$. Note that after translating the coordinate $t$ by $-s$, the pair $Q_1^s\subset\subset Q_2^s$ is identified with the fixed pair $[-1,1]\times E\subset\subset[-2,2]\times E.$  The Schauder constant depends only on these uniform ellipticity and coefficient bounds, the fixed cross-section $E$, and the fixed separation between the two cylinders, hence not on $s$.

Multiplying the estimate by $e^{\eta s}$, and using that $|t-s|\leq 2$ on $Q_2^s$, we have uniformly in $s$ that $\|e^{\eta s}u\|_{C^{m,\alpha}(Q_i^s)}$ and $\|e^{\eta t}u\|_{C^{m,\alpha}(Q_i^s)}$ are mutually bounded for $m\leq 4.$ This implies 

$$\|e^{\eta t}f^0\|_{C^{4,\alpha}(Q_1^s)}\leq C\left(\|e^{\eta t}L_{\omega_{HS}}f^0\|_{C^{0,\alpha}(Q_2^s)}+\|e^{\eta t}f^0\|_{L^2(Q_2^s)}\right).$$

Also note that near the divisor, the volume form $dV_{\omega_{HS}}$ is mutually bounded with $e^{-t}dt\,d\theta\,dV_E$, then
$$\|e^{\eta t}f^0\|_{L^2(Q_2^s)}\leq C\|f^0\|_{L^2_{\eta+\frac12,\delta}(\bl_0\C^n\backslash E)}.$$

Combining the above, we obtain

$$\|e^{\eta t}f^0\|_{C^{4, \alpha}(Q_1^s)}\le C\left(\|f^0\|_{L^2_{\eta+\frac{1}{2}, \dt}(\bl_0\C^n\backslash E)}+\|e^{\eta t}L_{\omega_{HS}}f^0\|_{C^{0,\alpha}(Q_2^s)}\right).$$

For $f^{\ne0}$, let $\Pi_0$ denote the fibrewise $S^1$-average and write $f^{\ne0}=(I-\Pi_0)f.$ Since $L_{\omega_{HS}}$ is $S^1$-invariant near the divisor, it commutes with $\Pi_0$, and hence

$$L_{\omega_{HS}}f^{\ne0}=(I-\Pi_0)L_{\omega_{HS}}f.$$

The zero-average estimate in \cite[Proposition~3.2]{Sek18}, together with its Schauder version \cite[Proposition~3.3]{Sek18}, gives a compact set $K_D\subset\subset \bl_0\C^n\backslash E$ and a constant $C$ such that

$$\|f^{\ne0}\|_{C^{4,\alpha}_{\eta,\delta}(\{t>\lambda+2\})}\leq C\left(\|L_{\omega_{HS}}f^{\ne0}\|_{C^{0,\alpha}_{\eta,\delta-4}(\{t>\lambda\})}+\|f^{\ne0}\|_{L^2_{\eta+\frac12,\delta}(\bl_0\C^n\backslash E)}+\|f^{\ne0}\|_{L^2(K_D)}\right).$$

With $\|f^{\ne0}\|_{L^2_{\eta+\frac12,\delta}}\leq C\|f\|_{L^2_{\eta+\frac12,\delta}}$, $\|f^{\ne0}\|_{L^2(K_D)}\leq C\|f^{\ne0}\|_{L^2_{\eta+\frac12,\delta}(\bl_0\C^n\backslash E)}$ and $\|L_{\omega_{HS}}f^{\ne0}\|_{C^{0,\alpha}_{\eta,\delta-4}}\leq C\|L_{\omega_{HS}}f\|_{C^{0,\alpha}_{\eta,\delta-4}}$, we obtain

$$\|f^{\ne0}\|_{C^{4,\alpha}_{\eta,\delta}(\{t>\lambda+2\})}\leq C\left(\|f\|_{L^2_{\eta+\frac12,\delta}}+\|L_{\omega_{HS}}f\|_{C^{0,\alpha}_{\eta,\delta-4}}\right).$$

Combining these two estimates gives the desired control near the divisor. 

For the control near the AE end, $f$ is viewed as a function on $[\mu, \infty)\times S^{2n-1}$, then for $s>\mu+2$, with similar arguments, we get 

$$\|r^{-\dt}f\|_{C^{4, \alpha}([s-1, s+1]\times S^{2n-1})}\lesssim\|f\|_{L^2_{\eta+\frac{1}{2}, \delta}(\bl_0\C^n\backslash E)}+\|r^{-\dt}L_{\omega_{HS}}f\|_{C^0([s-2, s+2]\times S^{2n-1})},$$

for details, we refer to Lemma 12.1.1 in \cite{Pac08}. 


Combining the estimates on the above three regions gives the desired estimate.
\end{proof}

The strategy to study function $f$ near the Poincar\'e divisor is to decompose the function $f=f^0+f^{\ne0}$ and deal with $f^0$ and $f^{\ne0}$ separately. The $f^0$ part is closely related to the product cylinder. We first introduce a key lemma concerning the $f^0$ term, which will be used later. Consider

$$C^{k, \al}_{\eta;0}([0, \infty)\times E)\coloneqq\{f\in C^{k, \al}_{\eta}([0, \infty)\times E), f(0, \cdot)=\partial_tf(0, \cdot)\equiv0\}.$$

Given $\{f_i\}$ a basis of $\ker \mathcal{D}^*_E\mathcal{D}_E$, and a cut-off function $\chi(t)$, define $\psi_{f_i}\coloneqq \chi(t)\pi^*f_i$. For $f\in \ker \mathcal D_E^*\mathcal D_E$, set

$$\psi_f:=\chi(t)\pi^*f.$$

From \cite{Auv14} Lemma 3.10, there exists a $\kappa>0$ such that for $\eta\in(0, \kappa)$, the following gives an isomorphism 

\begin{equation}\label{Auvray isomorphism}
    L_{\omega_{HS}}: C^{4, \al}_{\eta;0}([0, \infty)\times E)\oplus\text{Span}(\psi_{f_i})\rightarrow C^{0, \al}_{\eta}([0, \infty)\times E).
\end{equation}

Define the following modified H\"older space for functions on $Bl_0\C^n\backslash E$

$$\widetilde C^{0,\alpha}_{\eta,\delta-4}=C^{0,\alpha}_{\eta,\delta-4}\oplus\operatorname{Span}(\chi(t)\pi^*1),$$

where $\chi(t)$ is equal to one for $t$ sufficiently large and vanishes away from the cusp end. When restricting the modified space to the region near the divisor $E$, we obtain

$$\widetilde{C}^{0,\alpha}_{\eta}\coloneqq C^{0, \alpha}_{\eta}\oplus \text{Span}(\chi(t)\pi^* 1)\subset C^{0, \al}_{0}.$$

\begin{lem}\label{lemma S1 invariant inverse}
    The following operator 

$$C^{4,\alpha}_{\eta;0}([0, \infty)\times E)\oplus\text{Span}(\psi_{f_i})\oplus\mathbb{R}\rightarrow \widetilde{C}^{0,\alpha}_{\eta}([0, \infty)\times E)$$

given by

$$(\phi, f, \lambda)\mapsto\mathcal{D}^*_{\omega_{HS}}\mathcal{D}_{\omega_{HS}}(\phi+\psi_f+t\lambda)$$

has a bounded right inverse.
\end{lem}

\begin{proof}
     Given $\psi\in\widetilde{C}^{0,\alpha}_{\eta}([0, \infty)\times E)$, note
    
    $$\mathcal D^*_{\omega_{HS}}\mathcal D_{\omega_{HS}}(t\chi \pi^*1)=\chi\pi^*1+O(e^{-t}),$$
    
     hence there exists $\lambda\in\R$ such that $\psi-\mathcal D^*_{\omega_{HS}}\mathcal D_{\omega_{HS}}(t\chi \pi^*\lambda)\in C^{0, \alpha}_{\eta}$. Then, we conclude the surjectivity from (\ref{Auvray isomorphism}).

\end{proof}

For the $f^{\ne0}$ part, we introduce an essential lemma for the later study of $L_{\omega_{HS}}$. Consider the completion in the weighted Sobolev space $W^{4,2}_{\eta}(T_R)$ of the smooth functions that vanish on $\partial T_R$

$$W^{4,2}_{\eta;0}(T_R)\coloneqq\overline{
\left\{
\varphi\in C^\infty(\overline{T_R}) :
\varphi|_{\partial T_R}=0
\right\}
}^{\,W_{\eta}^{4,2}}.$$

\begin{lem}\label{lemma essential inequality}

Given $f\in W^{4,2}_{0;0}(T_R)$, consider its orthogonal decomposition $f=f^0+f^{\ne0}$. If $f^0=0$, then we have the following inequality

\begin{equation}\label{essential inequality}
    \|\mathcal{D}^*_{\omega_{HS}}\mathcal{D}_{\omega_{HS}}f\|_{L^2_{2}}\ge C\|f\|_{L^2}.
\end{equation}

\end{lem}

\begin{proof}

We first consider the Fourier decomposition of functions on $\bl_0\C^n\backslash E$ with respect to the $S^1$-action, whose infinitesimal generator is denoted by $\partial_z$. This decomposition will be used in the estimates below. We write

$$f=f^{\ne0}=\sum_{k\ne0}f_k$$

with $f_k$ satisfying $\mathcal{L}_{\partial_z}f_k=ikf_k$.

Recall that a $k$-th Fourier mode function $h$ on $Bl_0\C^n\backslash E$ corresponds to a section $\widehat h$ of $\mathcal{O}(k)$ over $\CP{n-1}$. In a local holomorphic trivialization $e_0$, this correspondence is given by

$$h(z,\xi)=u(z)\xi^k,\qquad \widehat h(z)=u(z)(e_0(z)^*)^k.$$

Thus the function $f_k$ corresponds to a section $\widehat f_k$ of $\mathcal{O}(k)$. We further decompose the $f_k$'s. To achieve this, we apply a similar technique as in \cite{HSVZ22}, Section~4.1. Using the spectral decomposition of the Dolbeault Laplacian $\Delta_k$ acting on sections of $\mathcal{O}(k)$, we write

$$\widehat f_k=\sum_{\lambda_0}\widehat f_{k,\lambda_0},\qquad \Delta_k\widehat f_{k,\lambda_0}=\lambda_0\widehat f_{k,\lambda_0}.$$

For each $\lambda_0$, let $f_{k,\lambda_0}$ denote the $k$-th Fourier mode function corresponding to the section $\widehat f_{k,\lambda_0}$ under the above identification. Therefore

$$f_k=\sum_{\lambda_0}f_{k,\lambda_0}.$$

After the above decompositions, we compute the local formula for the Lichnerowicz operator acting on a single component. Given

$$\mathcal{D}_{\omega_{HS}}^*\mathcal{D}_{\omega_{HS}}f=\frac{1}{2}\Delta^2_{\omega_{HS}}f+\langle \rho_{\omega_{HS}}, dd^cf\rangle,$$

the first step is to derive the local formula for $dd^cf$. Recall

$$\omega_{HS}=(1+\frac{\tau}{2})p^*\omega_{FS}+d\tau\wedge\Theta,$$

where $\Theta=\varphi^{-1}d^c\tau$. We show that $\Theta$ satisfies 

\begin{equation}\label{eqn: theta partial tau vanishes}
    \Theta(\partial_{\tau})=0.
\end{equation}

It is equivalent to show $d^c\tau(\partial_{\tau})=d\tau(J\partial_{\tau})=0$. For $r=\log|z|_h$, from the momentum construction, we know $\varphi dr=d\tau$, then $d\tau(J\partial_{\tau})=\varphi dr(\varphi^{-1}J\partial_r)=0$. 

With Equation~\eqref{eqn: theta partial tau vanishes}, we pick a chart on $E$ locally, and write $\Theta=dz+\theta$, where $\theta$ is a connection $1$-form depending on $E$. Further we obtain

$$df=d_E^{\theta}f+\partial_{\tau}fd\tau+\partial_zf\Theta,$$

where $d_E^{\theta}\coloneqq d_E-\partial_zf\cdot\theta$. It is an operator in the divisor direction. Then

$$d^cf=d_E^{\theta, c}f+\partial_{\tau}f\varphi \Theta-\partial_zf\varphi^{-1}d\tau,$$

further we get 

\begin{align*}
    dd^cf=d_E^{\theta}d_E^{\theta,c}f-d_E^{\theta}(\partial_zf\varphi^{-1})\wedge d\tau+d\tau\wedge\partial_{\tau}(d_E^{\theta,c}f)+\Theta\wedge(\partial_zd_E^{\theta,c}f)+\varphi^{-1}\partial_z^2fd\tau\wedge\Theta+\partial_{\tau}f\varphi d\Theta\\
    +\partial_{\tau}(\partial_{\tau}f\varphi^{-1})d\tau\wedge\Theta+d_E^{\theta}(\partial_{\tau}f\varphi)\wedge\Theta.
\end{align*}

The second step is to compute the formula for $\Delta_{HS}$. For any $k$-th Fourier mode function $h$, define

$$\Delta_{FS}^{\theta}h\coloneqq(n-1)\frac{d_E^{\theta}d_E^{\theta,c}h\wedge p^*\omega_{E}^{n-2}}{p^*\omega_{E}^{n-1}}.$$

Pick a local chart, with coordinates $x_i,y_i$ on $E$ and coordinate $\xi$ on the fiber. If $h(z,\xi)=u(z)\xi^k$ and $\widehat h(z)=u(z)(e_0(z)^*)^k$, then locally

$$\Delta_{FS}^{\theta}h=\frac{1}{4}\sum_i\left(\frac{\partial^2 h}{\partial x_i^2}+\frac{\partial^2 h}{\partial y_i^2}\right).$$

Applying this to $h=f_{k,\lambda_0}$ gives

$$\Delta_{FS}^{\theta}f_{k,\lambda_0}=\frac{1}{4}\sum_i\left(\frac{\partial^2 f_{k,\lambda_0}}{\partial x_i^2}+\frac{\partial^2 f_{k,\lambda_0}}{\partial y_i^2}\right).$$

Applying the Bochner-Kodaira-Nakano formula to $\Delta_k$, and writing $\Delta_k^{\partial}$ for the $\partial$-Laplacian, we have $\Delta_k=\Delta_k^{\partial}+k(n-1)$, and hence $\lambda_0\ge k(n-1)$. Since $\Delta_k\widehat f_{k,\lambda_0}=\lambda_0\widehat f_{k,\lambda_0}$, we deduce that

$$\Delta_{FS}^{\theta}f_{k,\lambda_0}=\left(-\lambda_0+\frac{k(n-1)}{2}\right)f_{k,\lambda_0}.$$

For notational convenience, set

$$\lambda\coloneqq -\lambda_0+\frac{k(n-1)}{2},\qquad f_{k,\lambda}\coloneqq f_{k,\lambda_0}.$$

With this convention,

$$\Delta_{FS}^{\theta}f_{k,\lambda}=\lambda f_{k,\lambda}.$$

Note

$$d\Theta=d(d^c\tau\cdot\varphi^{-1})=\varphi^{-1}\cdot dd^c\tau+\partial_{\tau}(\varphi^{-1})d\tau\wedge d^c\tau.$$

We recall from the momentum construction that given a function $\psi(\tau)$ on the total space of line bundle, we have

$$dd^c\psi(\tau)=(\varphi\psi')(\tau)\frac{1}{2}\pi^*\omega_{FS}+\frac{1}{\varphi}(\varphi\psi')'(\tau)d\tau\wedge d^c\tau.$$

Thus 

$$dd^c\tau=\frac{1}{2}\varphi\pi^*\omega_{FS}+\varphi^{-1}\varphi_{\tau}d\tau\wedge d^c\tau.$$ 

In particular, we note $d\Theta$ does not have contribution of mixed term of fiber and base. Write $\theta=\theta^{x_i}dx_i+\theta^{y_i}dy_i$, we have

$$\partial_{\tau}(\theta^{x_i})=\partial_{\tau}(\theta^{y_i})=0.$$

From this we know $\theta$ and thus $d_E^{\theta}, d_E^{\theta, c}$ are invariant under $\tau$. From now on we abbreviate $f_{k, \lambda}$ as $f$ for simplification. Then 

$$\Delta_{\omega_{HS}}f=n\displaystyle\frac{dd^cf\wedge\omega_{HS}^{n-1}}{\omega_{HS}^n}=\displaystyle\frac{\lambda}{1+\frac{\tau}{2}}f-\varphi^{-1}k^2f+\partial_{\tau}(\varphi\partial_{\tau}f)+\frac{n-1}{2(1+\frac{\tau}{2})}\varphi\partial_{\tau}f.$$

The third step is to compute the Ricci curvature of $\omega_{HS}$. From the momentum construction, the expression for the Ricci curvature is

$$\rho=\pi^*\omega_{FS}+\frac{1}{2Q}(\varphi Q)'(\tau)\pi^*\gamma-\frac{1}{2}[\frac{1}{Q}(\varphi Q)']'(\tau)d\tau\wedge \Theta,$$

where $\gamma=-\ddb\log h$ for hermitian metric $h$ on the line bundle. In our case, plugging in $Q(\tau)=(1+\frac{\tau}{2})^{n-1}$ and $\varphi(\tau)=\frac{8}{n}\left(1+\frac{\tau}{2}-\frac{1+\frac{n}{2}\tau}{(1+\frac{\tau}{2})^{n-1}}\right)$, we obtain

$$\rho_{\omega_{HS}}=\frac{1}{(1+\frac{\tau}{2})^{n-1}}\pi^*\omega_{E}-\frac{1}{(1+\frac{\tau}{2})^n}d\tau\wedge\Theta,$$

then

$$\langle \rho_{\omega_{HS}}, dd^cf\rangle=\displaystyle\frac{1}{(1+\frac{\tau}{2})^{n+1}}(\lambda f+\frac{1}{2}\varphi f_{\tau})+\displaystyle\frac{1}{(1+\frac{\tau}{2})^n}(\varphi^{-1}k^2f-\partial_{\tau}(\varphi f_{\tau})).$$



From all above three steps, we obtain an expression of the Lichnerowicz operator acting on $f_{k,\lambda}$. We are ready to prove the desired inequality. Given $f\in W^{4,2}_{0;0}$, we show that $f_{k, \lambda}\in W^{4,2}_{0;0}$. 

Note that the above decompositions preserve the boundary condition. Let $P_{k,\lambda}$ be the projection onto the $(k,\lambda)$-component. For smooth \(\varphi\) with \(\varphi|_{\partial T_R}=0\), the projection $P_{k,\lambda}$ acts only in the variables tangent to $\partial T_R$, and hence

\[
(P_{k,\lambda}\varphi)|_{\partial T_R}=P_{k,\lambda}(\varphi|_{\partial T_R})=0.
\]

Since \(P_{k,\lambda}\) is bounded on \(W_\eta^{4,2}(T_R)\), it follows by
taking the \(W_\eta^{4,2}\)-closure that
\[
P_{k,\lambda}\bigl(W_{\eta;0}^{4,2}(T_R)\bigr)\subset W_{\eta;0}^{4,2}(T_R).
\]

Therefore each component $f_{k,\lambda}$ also lies in $W_{\eta;0}^{4,2}(T_R)$, and the integrations by parts below are justified
for each component. For the following we denote $f=f_{k, \lambda}$. 

\begin{align*}
    \displaystyle\int_{T_R\backslash E}2\mathcal{D}^*\mathcal{D} f\cdot f\;dvol
    &=\displaystyle\int(\Delta_{\omega_{HS}}f)^2\;dvol+\displaystyle\int2\langle \rho_{\omega_{HS}}, dd^cf\rangle f\;dvol\\
    &=\displaystyle\int\left(-\varphi^{-1}k^2f+\varphi_{\tau}f_{\tau}+\frac{n-1}{2(1+\frac{\tau}{2})}\varphi f_{\tau}+\varphi f_{\tau\tau}\right)^2+2f\cdot f_{\tau}\cdot\varphi_{\tau}\left(\frac{\lambda}{1+\frac{\tau}{2}}-\frac{1}{(1+\frac{\tau}{2})^n}\right)\\
    &+\displaystyle\int f\cdot f_{\tau}\cdot\varphi\left(\frac{(n-1)\lambda}{(1+\frac{\tau}{2})^2}+\frac{1}{(1+\frac{\tau}{2})^{n+1}}\right)+\displaystyle\int2f\cdot f_{\tau\tau}\cdot\varphi\left(\frac{\lambda}{1+\frac{\tau}{2}}-\frac{1}{(1+\frac{\tau}{2})^n}\right)\\
    &+\left(2\varphi^{-1}k^2\left(\frac{1}{(1+\frac{\tau}{2})^n}-\frac{\lambda}{1+\frac{\tau}{2}}\right)+\frac{2\lambda}{(1+\frac{\tau}{2})^{n+1}}+\frac{\lambda^2}{(1+\frac{\tau}{2})^2}\right)f^2\\
    &=\displaystyle\int\left(-\varphi^{-1}k^2f+\varphi_{\tau}f_{\tau}+\frac{n-1}{2(1+\frac{\tau}{2})}\varphi f_{\tau}+\varphi f_{\tau\tau}\right)^2+f\cdot f_{\tau}\cdot \varphi\left(\frac{n\lambda}{(1+\frac{\tau}{2})^2}-\frac{n-1}{(1+\frac{\tau}{2})^{n+1}}\right)\\
    &+\displaystyle\int\left(2\varphi^{-1}k^2\left(\frac{1}{(1+\frac{\tau}{2})^n}-\frac{\lambda}{1+\frac{\tau}{2}}\right)+\frac{2\lambda}{(1+\frac{\tau}{2})^{n+1}}+\frac{\lambda^2}{(1+\frac{\tau}{2})^2}\right)f^2+\left(\frac{2\varphi}{(1+\frac{\tau}{2})^n}-\frac{2\lambda\varphi}{1+\frac{\tau}{2}}\right) f_{\tau}^2\\
    &\ge \displaystyle\int C \tau^{-2}f^2\ge\displaystyle\int C' e^{2t}f^2\;dvol
\end{align*}

Applying Cauchy--Schwartz inequality to the above, we obtain the desired inequality for $f_{k, \lambda}$. Using the orthogonality of the eigenfunctions corresponding to different eigenvalues, we obtain the inequality for $f$.

\end{proof}

We are now ready to state the main result of this section.

\begin{thm}\label{theorem bounded inverse for cusp}
There exists a constant $\kappa>0$ such that for $\eta\in(0, \kappa)$ and $\delta\in(0,1)$, the operator

$$G_{\omega_{HS}}: C^{4,\alpha}_{\eta, \delta}(Bl_0\mathbb{C}^n\backslash E)\oplus\ker\mathcal{D}^*_E\mathcal{D}_E\oplus\mathbb{R}\rightarrow \widetilde{C}^{0,\alpha}_{\eta, \delta-4}(Bl_0\mathbb{C}^n\backslash E)$$

given by

$$(\phi, f, \lambda)\mapsto\mathcal{D}^*_{\omega_{HS}}\mathcal{D}_{\omega_{HS}}(\phi+\psi_f+t\lambda)$$

has a bounded right inverse.
\end{thm}

\begin{proof}
Consider a cut-off function $\gamma_1$ supported on a tubular neighborhood of radius $R$ near E, denoted by $T_R\backslash E$, and let $\nabla\gamma_1$ supported on $T_R\backslash T_{\frac{R}{2}}$ with $\|\gamma_1\|_{C^{4, \al}}\le c$. Define $\gamma_2\coloneqq 1-\gamma_1$ satisfying $\|\gamma_2\|_{C^{4, \al}}\le c$. Consider $\varphi\in \tilde{C}^{0, \al}_{\eta, \delta-4}$, then $\gamma_1\varphi$ is viewed as a function on $T_R\backslash E$ and $\gamma_2\varphi$ is viewed as a function on $Bl_0\C^n\backslash Bl_0B_R$.

\medskip

Given $\gamma_2\varphi$ supported on a AE manifold $\bl_0\C^n\backslash\bl_0B_R$, under the weight $\delta\in(0, 1)$, in order to show the map $\mathcal{D}^*\mathcal{D}: C^{4, \al}_{\delta}\rightarrow C^{0, \al}_{\delta-4}$ is surjective, it suffices to show its adjoint operator $\mathcal{D}^*\mathcal{D}: C^{4, \al}_{4-2n-\delta}\rightarrow C^{0, \al}_{-2n-\delta}$ is injective. To find the kernel $\mathcal{D}^*\mathcal{D}f=0$, we first improve the decay rate of the solution at infinity. Note $4-2n-\delta<0$, we use standard biharmonic analysis for Euclidean space and obtain $f\in C^{4, \al}_{4-2n}$ for $n\ge 3$, thus integration by parts is justified, which implies $\mathcal{D}f=0$. This gives rise to a holomorphic vector field and thus holomorphic function on the complement of $\C^n$ of a ball. Then Hartogs's extension theorem implies the holomorphic function is extended to all of $\C^n$, and hence is a constant. Note again $4-2n-\delta<0$, we know the constant is zero, thus we conclude the triviality of the kernel. Then there exists $P_2(\gamma_2\varphi)$ such that $\mathcal{D}^*\mathcal{D}(P_2(\gamma_2\varphi))=\gamma_2\varphi$ and satisfies

$$\|P_2(\gamma_2\varphi)\|_{C^{4, \al}_{\delta}}\lesssim\|\gamma_2\varphi\|_{C^{0, \al}_{\delta-4}}\lesssim\|\varphi\|_{\tilde{C}^{0, \al}_{\eta, \delta-4}}.$$

Given $\gamma_1\varphi$ supported on $T_R\backslash E$, satisfying $\gamma_1\varphi|_{\partial T_R}=0$. For a function $f\in C^{k, \al}(T_R\backslash E)$, consider the orthogonal decomposition $f=f^0+f^{\ne0}$, where $f^0$ is invariant under the holomorphic $S^1$ action on $(Bl_0\C^n\backslash E, \omega_{HS})$, and $f^{\ne0}$ has zero mean on each fiber. Note $\omega_{HS}$ is $S^1$-invariant, and thus the associated Lichnerowicz operator maps $S^1$-invariant part to $S^1$-invariant and zero mean part to zero mean part. To solve $L_{\omega_{HS}}\phi=\gamma_1\varphi$ on $T_R\backslash E$, we solve these two parts separately.

\medskip

For $(\gamma_1\varphi)^0$, we view it as a function on $T_R/S^1=[0,\infty)\times E$. The quotient metric is of the form $dt^2+\omega_{FS}+O(e^{-t})$, with volume form mutually bounded by $e^{-t}dt\,dvol_E$, where $t=\log(-\log|z|_h^2)+O(e^{-t})$. The Lichnerowicz operator has the expression
$$L_{\omega_{HS}}=\mathcal D^*_{\omega_{HS}}\mathcal D_{\omega_{HS}}=\frac{1}{2}(\partial_z-\partial_z^2)^2+\mathcal{D}^*_E\mathcal{D}_E+\Delta_{FS}(\partial_z-\partial_z^2)+(\partial_z-\partial_z^2)+O(e^{-t}).$$
The solvability follows from Lemma~\ref{lemma S1 invariant inverse}.

It remains to deal with the function $(\gamma_1\varphi)^{\ne0}$.

We construct the inverse of $(\gamma_1\varphi)^{\ne0}$ under the operator $\mathcal{D}^*\mathcal{D}$ by first doing cut-off so that it supports in a compact annulus, then we find its inverse under $\mathcal{D}^*\mathcal{D}$ on each annulus, and finally we glue them together using partition of unity. This strategy is first explored by Biquard in \cite{biquard1997fibres}. The inequality (\ref{essential inequality}) proved in Lemma~\ref{lemma essential inequality} is used to improve the regularity of the inverse. 

\medskip

Let $(\zeta_A)_{A\in\mathbb{Z}_+}$ be a partition of unity to the intervals $(A-1, A+1)$ in coordinate $t$. Let $U_A$ be the annulus whose coordinate $t$ lives in $(A-1, A+1)$. From (\ref{essential inequality}), we know the map 

$$\mathcal{D}^*\mathcal{D}: W^{4,2}_0(U_A)\subset L^2(U_A)\rightarrow L^2(U_A)$$

is injective. Here, $W^{4,2}_0(U_A)$ is the closure of $C^{\infty}_c(U_A)$ in $W^{4,2}(U_A)$. Note the operator above is densely defined, closed, and self-adjoint, then from \cite{brezis2011functional} Theorem 2.19, we know $\Ima\mathcal{D}^*\mathcal{D}$ being closed is equivalent to $\Ima\mathcal{D}^*\mathcal{D}=(\ker\mathcal{D}^*\mathcal{D})^{\perp}$. 

We first show the following elliptic a priori estimates

$$\|f\|_{W^{4,2}(U_A)}\lesssim \|\mathcal{D}^*\mathcal{D}f\|_{L^2(U_A)}+\|f\|_{L^2(U_A)}, \quad f\in W^{2,2}_0.$$

We consider $\widetilde U_A=\{A-2<t<A+2\}$. For $f\in C_c^\infty(U_A)$, extend $f$ by zero to $\widetilde U_A$. By the standard interior elliptic estimate for the fourth-order elliptic operator $\mathcal{D}^*\mathcal{D}$, applied to the fixed inclusion $U_A\subset\subset\widetilde U_A$, there is a constant $C$, independent of $f$, such that

$$\|f\|_{W^{4,2}(U_A)}\leq C\left(\|\mathcal{D}^*\mathcal{D} f\|_{L^2(\widetilde U_A)}+\|f\|_{L^2(\widetilde U_A)}\right).$$

Since the zero extension vanishes on $\widetilde U_A\setminus U_A$, this becomes

$$\|f\|_{W^{4,2}(U_A)}\leq C\left(\|\mathcal{D}^*\mathcal{D} f\|_{L^2(U_A)}+\|f\|_{L^2(U_A)}\right).$$

Passing to the closure of $C_c^\infty(U_A)$ in $W^{4,2}(U_A)$, and using the continuity of $\mathcal{D}^*\mathcal{D}:W^{4,2}(U_A)\to L^2(U_A)$, we obtain the same estimate for $f\in W^{4,2}_0(U_A)$. In fact, with same argument as in Theorem~\ref{Schauder estimates}, we see this constant $C$ is also independent from $A$.

Then from standard arguments we know $\Ima\mathcal{D}^*\mathcal{D}$ is closed, see \cite{xuzheng24} Lemma 4.3, \cite{Pac08} Theorem 9.2.1. Hence the map is surjective, for any $B\in \Z_+$, we use $(\mathcal{D}^*\mathcal{D})^{-1}(\xi_Bg)$ to denote the inverse. On the annulus $U_A$, let $1_A$ denote the function on it which is constantly equal to $1$. We consider the following weighted estimate 

\begin{align*}
    \displaystyle\int_{U_A}|\mathcal{D}^*\mathcal{D}f|^2e^{(2\eta-3) t}&\ge min\{e^{(2\eta-3)(A-1)},e^{(2\eta-3)(A+1)}\}\displaystyle\int_{U_A} |\mathcal{D}^*\mathcal{D}f|^2\\
    &= e^{-2|2\eta-3|}max\{e^{(2\eta-3)(A-1)},e^{(2\eta-3)(A+1)}\}\displaystyle\int_{U_A} |\mathcal{D}^*\mathcal{D}f|^2\\
    &\ge e^{-2|2\eta-3|}max\{e^{(2\eta-3)(A-1)},e^{(2\eta-3)(A+1)}\}\displaystyle\int_{U_A}|f|^2e^{4t}\\
    &\ge e^{-2|2\eta-3|}\displaystyle\int_{U_A}|f|^2e^{(2\eta+1)t}.
\end{align*}

Note that $e^{-2|2\eta-3|}$ is a constant independent of $A$, then any $A, B\in\Z_+$, applying the above inequality, we get

\begin{align*}
\|\zeta_A(\mathcal{D}^*\mathcal{D})^{-1}(\zeta_Bg)\|_{L^2_{\eta+\frac{1}{2}}}&\le e^{(\eta-\eta_1)A}\|\zeta_A(\mathcal{D}^*\mathcal{D})^{-1}(\zeta_Bg)\|_{L^2_{\eta+\frac{1}{2}}}\\
&\lesssim e^{(\eta-\eta_1)A}\|1_A\cdot(\mathcal{D}^*\mathcal{D})^{-1}(\zeta_Bg)\|_{L^2_{\eta+\frac{1}{2}}}\\
&\lesssim e^{(\eta-\eta_1)A}\|\zeta_Bg\|_{L^2_{\eta_1-\frac{3}{2}}}\\
&\lesssim e^{(\eta-\eta_1)(A-B)}\|\zeta_Bg\|_{L^2_{\eta_1-\frac{3}{2}}}
\end{align*}

For each $A$, choose $\eta_1$ such that $\eta-\eta_1=-\epsilon\cdot sgn(A-B)$, then the above expression becomes 

$$\|\zeta_A(\mathcal{D}^*\mathcal{D})^{-1}(\zeta_Bg)\|_{L^2_{\eta+\frac{1}{2}}}\lesssim e^{-\epsilon|A-B|}\|\zeta_Bg\|_{L^2_{\eta_1-\frac{3}{2}}}$$

Combining the above, we have the following estimate

\begin{align*}
\|\displaystyle\sum_A\displaystyle\sum_B\zeta_A(\mathcal{D}^*\mathcal{D})^{-1}(\zeta_Bg)\|_{L^2_{\eta+\frac{1}{2}}}
&\le\left(\displaystyle\sum_A\|\displaystyle\sum_B\zeta_A(\mathcal{D}^*\mathcal{D})^{-1}(\zeta_Bg)\|^2_{L^2_{\eta+\frac{1}{2}}}\right)^{\frac{1}{2}}\\
&\le\left(\displaystyle\sum_A\left(\displaystyle\sum_B\|\zeta_A(\mathcal{D}^*\mathcal{D})^{-1}(\zeta_Bg)\|_{L^2_{\eta+\frac{1}{2}}}\right)^2\right)^{\frac{1}{2}}\\
&\lesssim \left(\displaystyle\sum_A\left(\displaystyle\sum_B e^{-\ep|A-B|}\|\zeta_Bg\|_{L^2_{\eta-\frac{3}{2}}}\right)^2\right)^{\frac{1}{2}}\\
&\lesssim \left(\displaystyle\sum_B\|\zeta_Bg\|^2_{L^2_{\eta-\frac{3}{2}}}\right)^{\frac{1}{2}}\le \|g\|_{L^2_{\eta-\frac{3}{2}}}\lesssim\|g\|_{C^{0, \al}_{\eta}}
\end{align*}

From elliptic regularity Theorem~\ref{Schauder estimates}, we conclude $\displaystyle\sum_A\displaystyle\sum_B\zeta_A(\mathcal{D}^*\mathcal{D})^{-1}(\zeta_Bg)\in C^{4, \al}_{\eta}$. Combining this construction with Lemma~\ref{lemma S1 invariant inverse}, we obtain $P_1(\gamma_1\varphi)$ satisfying $G_{\omega_{HS}}(P_1(\gamma_1\varphi))=\gamma_1\varphi$ and

$$\|P_1(\gamma_1\varphi)\|_{C^{4,\al}_{\eta}\times\ker\mathcal{D}_E^*\mathcal{D}_E\times\R}\lesssim\|\varphi\|_{\tilde{C}^{0,\al}_{\eta,\delta-4}}$$

Consider for $i=1,2$ the cutoff functions $\beta_i$ with $\beta_i=1$ on the support of $\gamma_i$, and $\nabla\beta_1$ supports on $T_{2R}\backslash T_R$, $\nabla\beta_2$ supports on $T_{\frac{R}{2}}\backslash T_{\frac{R}{4}}$, further, $\|\nabla\beta_1\|_{C^{3,\al}}\le\frac{C}{\log R}$ and $\|\nabla\beta_2\|_{C^{3, \al}}\le\frac{C}{\log R}$. This can be achieved by defining $\beta_1\coloneqq\beta\left(\frac{|z|}{\log R}\right)$ where $\beta(r)=1$ for $r<\frac{R}{\log R}$ and $\beta(r)=0$ for $r>\frac{2R}{\log R}$; $\beta_2\coloneqq\tilde{\beta}\left(\frac{|z|}{\log R}\right)$ where $\tilde{\beta}(r)=1$ for $r>\frac{R}{2\log R}$ and $\tilde{\beta}=0$ for $r<\frac{R}{4\log R}$. Define 

$$P(\varphi)=\beta_1P_1(\gamma_1\varphi)+\beta_2P_2(\gamma_2\varphi)$$

First we show that $\varphi\mapsto P(\varphi)$ gives an approximate inverse, i.e., 

\begin{equation}\label{inverse bound inequality}
    \|G_{\omega_{HS}}(P(\varphi))-\varphi\|_{\tilde{C}^{0,\al}_{\eta, \delta-4}}\le\frac{1}{2}\|\varphi\|_{\tilde{C}^{0,\al}_{\eta, \delta-4}}.
\end{equation}

It reduces to show 

$$\|G_{\omega_{HS}}(\beta_iP_i(\gamma_i\varphi))-\gamma_i\varphi\|\le\frac{1}{4}\|\varphi\|.$$

Write 

$$G_{\omega_{HS}}(\beta_iP_i(\gamma_i\varphi))-\gamma_i\varphi=\beta_i\gamma_i\varphi-\gamma_i\varphi+D^3(\nabla\beta_i\star P_i(\gamma_i\varphi)),$$

where $D^3$ denotes a third-order differential operator, and $\star$ a bilinear operator. Note 

$$\|D^3(\nabla\beta_1\star P_1(\gamma_1\varphi))\|_{C^{0, \al}_{\eta}}\lesssim\|\nabla\beta_1\|_{C^{3,\al}}\|P_1(\gamma_1\varphi)\|_{C^{3,\al}_{\eta}}=o(1)\|\varphi\|_{C^{0, \al}_{\eta}},$$

$$\|D^3(\nabla\beta_2\star P_2(\gamma_2\varphi))\|_{C^{0, \al}_{\delta-4}}\lesssim\|\nabla\beta_2\|_{C^{3,\al}_{-1}}\|P_2(\gamma_2\varphi)\|_{C^{3,\al}_{\delta}}=o(1)\|\varphi\|_{C^{0, \al}_{\delta-4}}.$$

Here $\nabla\beta_2$ is supported on $T_{\frac{R}{2}}\backslash T_{\frac{R}{4}}$, and $\|\nabla\beta_2\|_{C^{3,\al}_{-1}}\le\frac{C}{\log R}$, thus proving (\ref{inverse bound inequality}). From (\ref{inverse bound inequality}), we see there is a uniformly bounded inverse $(G_{\omega_{HS}}\circ P)^{-1}$, thus $G_{\omega_{HS}}$ has a right inverse $P\circ(G_{\omega_{HS}}\circ P)^{-1}$.

\end{proof}

\section{Existence of solution}\label{section 8}

To obtain a cscK metric on $Bl_pX\backslash E$, we want to solve a boundary value problem for the cscK equation on the two pieces $X_{r_{\epsilon}}$ and $Bl_0 B_{R_{\ep}}\backslash E$. The arguments for the piece $X_{r_{\epsilon}}$ are the same as Arezzo-Pacard's case. We can directly take from \cite{arezzo2006blowing}. 

\begin{thm}(\cite{arezzo2006blowing}, Proposition 6.1)\label{theorem cscK on punctured base}
    Fix a constant $\Lambda>0$ large enough, there exist $c_{\Lambda}>0$, $\epsilon_{\Lambda}>0$ such that $\forall\epsilon\in(0, \epsilon_{\Lambda})$, and all boundary data $h\in C^{4, \al}(\partial B_1), k\in C^{2, \al}(\partial B_1)$ satisfying

    \begin{equation}\label{conditions for h,k}
        \|h\|_{C^{4, \al}(\partial B_1)}\le \Lambda r_{\epsilon}^4, \quad \|k\|_{C^{2, \al}(\partial B_1)}\le\Lambda r_{\epsilon}^4,
    \end{equation}

    there exists $\varphi_{\epsilon, h,k}$ such that the K\"ahler form $\omega_{\epsilon, h,k}$ defined by Equation~\eqref{metric on base piece} on $X_{r_{\epsilon}}$ has constant scalar curvature $s(\omega_X)+\nu_{\epsilon,h,k}$. If we use $\hat{\varphi}$ to denote the component in $\varphi$ which is defined on $X_p$, then 

    $$\|\hat{\varphi}_{\ep, h,k}|_{B_{2r_{\ep}}\backslash B_{r_{\ep}}}(r_{\ep}\cdot)-H^0_{h,k}\|_{C^{4, \al}(B_2\backslash B_1)}\le c_{\Lambda}r_{\ep}^{2n+\dt},\quad |\nu_{\epsilon, h,k}|\le c_{\Lambda}r_{\ep}^{2n}.$$

    Moreover,

    \begin{equation}\label{nonlinear estimate for base 01}
        \|\hat{\varphi}_{\ep, h,k}-\hat{\varphi}_{\ep, h',k'}|_{B_{2r_{\ep}}\backslash B_{r_{\ep}}}(r_{\ep}\cdot)-H^0_{h,k}\|_{C^{4, \al}(B_2\backslash B_1)}\le c_{\Lambda}r_{\ep}^{2n-4+\dt}\|h-h', k-k'\|,
    \end{equation}

    \begin{equation}\label{nonlinear estimate for base 02}
        |\nu_{\epsilon, h,k}-\nu_{\epsilon, h',k'}|\le c_{\Lambda}r_{\ep}^{2n-4}\|h-h', k-k'\|.
    \end{equation}
    
\end{thm}

The main difference lies in the piece of $Bl_0\C^n\backslash E$, which we discuss as follows.

\begin{thm} \label{theorem modified cscK}
   Fix a constant $\Lambda>0$ large enough, there exist $c>0$ independent of $\Lambda$ and $\epsilon_{\Lambda}>0$ such that $\forall\ep\in(0, \ep_{\Lambda})$, and for all $\nu\in\R$ and boundary data $\tilde{h}\in C^{4, \al}(\partial B_1), \tilde{k}\in C^{2, \al}(\partial B_1)$ satisfying 

   \begin{equation}\label{conditions for tilde h,k}
       |\nu|\le |s(\omega_X)|+1, \quad \|\tilde{h}\|_{C^{4, \al}(\partial B_1)}\le \Lambda R_{\ep}^{3-2n}, \quad \|\tilde{k}\|_{C^{2, \al}(\partial B_1)}\le \Lambda R_{\ep}^{3-2n},
   \end{equation} 

   there exist $\phi_{\ep}\in C^{4, \al}_{\eta, \dt}$, constant $\lambda_{\ep}\in\R$ and $f_{\ep}\in\ker\mathcal{D}^*_E\mathcal{D}_E$ such that the scalar curvature of the K\"ahler form
   
   $$\omega_{\epsilon, \tilde{h},\tilde{k}, HS}=\omega_{HS}+\displaystyle\sqrt{-1}\partial\bar{\partial}(H^1_{\ep, \tilde{h}, \tilde{k}}+\phi_{\ep}+\lambda_{\ep}t)$$
   
   is 
   
      $$\ep^2\nu+\mathcal D^*_{\omega_{HS}}\mathcal D_{\omega_{HS}}(\psi_{f_{\ep}})+\frac{1}{2}\langle \nabla\mathcal D^*_{\omega_{HS}}\mathcal D_{\omega_{HS}}(\psi_{f_{\ep}}), \nabla(H^1_{\ep, \tilde{h}, \tilde{k}}+\phi_{\ep}+\lambda_{\ep} t)\rangle.$$
   
   Further, 

   $$\|\phi_{\ep}(R_{\ep}\cdot)\|_{C^{4, \al}(B_2\backslash B_1)}\le cR_{\ep}^{3-2n}, \quad |\lambda_{\ep}|+\|f_{\ep}\|\le cR_{\ep}^{3-2n-\dt}.$$

   If $\phi_{\ep}, \lambda_{\ep}$ are determined by boundary data $\tilde{h}, \tilde{k}$ and $\phi'_{\ep}, \lambda'_{\ep}$ are determined by boundary data $\tilde{h}', \tilde{k}'$, then

   \begin{equation}\label{equation n ge 3 nonlinear estimate}
       \|f_{\ep}-f_{\ep}'\|+|\lambda_{\ep}-\lambda'_{\ep}|+\|\phi_{\ep}(R_{\ep}\cdot)-\phi_{\ep}'(R_{\ep}\cdot)\|\le c_{\Lambda}(R_{\ep}^{\dt-1}\|(\tilde{h}-\tilde{h}',\tilde{k}-\tilde{k}')\|+R_{\ep}^{3-2n}|\nu-\nu'|);
   \end{equation}
\end{thm}

\begin{proof}
    In the proof, we omit the subscript $\ep$ for $\phi_{\ep}, \lambda_{\ep}, f_{\ep}$ for simplicity. We want to solve the equation

    \begin{equation}\label{equation modified cscK}    s(\omega_{HS}+\displaystyle\sqrt{-1}\partial\bar{\partial}(H^1_{\ep, \tilde{h}, \tilde{k}}+\phi+\lambda t))=\ep^2\nu+\mathcal D^*_{\omega_{HS}}\mathcal D_{\omega_{HS}}(\psi_{f})+\frac{1}{2}\langle \nabla\mathcal D^*_{\omega_{HS}}\mathcal D_{\omega_{HS}}(\psi_{f}), \nabla(H^1_{\ep, \tilde{h}, \tilde{k}}+\phi+\lambda t)\rangle,
    \end{equation}

    for a function $\phi$, a constant $\lambda$ and a holomorphy potential $f$ on $E$. If we separate the equation into linear and non-linear terms,  we get
    
    $$\begin{aligned}
\mathcal{D}^*_{\omega_{HS}}\mathcal{D}_{\omega_{HS}}(\phi+\psi_f+t\lambda)
=&Q_{\omega_{HS}}(H^1_{\ep, \tilde{h}, \tilde{k}}+\phi+\lambda t)-\mathcal D^*_{\omega_{HS}}\mathcal D_{\omega_{HS}}(H^1_{\ep, \tilde{h}, \tilde{k}})-\ep^2\nu\\
&-\frac{1}{2}\langle \nabla\mathcal D^*_{\omega_{HS}}\mathcal D_{\omega_{HS}}(\psi_{f}), \nabla(H^1_{\ep, \tilde{h}, \tilde{k}}+\phi+\lambda t)\rangle.
\end{aligned}$$

   We denote the RHS by $Q_{\ep}$ and we show that $Q_\ep\in \widetilde C^{0,\alpha}_{\eta,\delta-4}$. For the $Q_{\omega_{HS}}$ term, note that $H^1_{\ep,\tilde h,\tilde k}$ is supported away from the cusp end, so for $t$ sufficiently large the relevant potential is $\phi+\lambda t$. Since $\phi\in C^{4,\alpha}_{\eta,\delta}$, its derivatives are of order $O(e^{-\eta t})$. Also note that the leading part of $t$ is independent of the divisor direction and $\sqrt{-1}\partial\bar\partial t$ changes to leading order only the coefficient of the normal cusp factor. From the asymptotic expansion of the Hwang-Singer metric near $E$, we obtain that the perturbed metric has asymptotic behavior, 

$$a_\lambda(dt^2+e^{-2t}\Theta^2)+\pi^*g_{FS}+O(e^{-\eta' t}),$$

for some positive constant $a_\lambda$ and some $\eta'>0$. The leading term has a constant coefficient $a_\lambda$ in the normal cusp direction, and the divisor factor is the Fubini-Study cscK metric. Therefore the scalar curvature of the leading term has a constant limit as $t\to\infty$. Since scalar curvature is a smooth local expression in the metric, its inverse, and two derivatives of the metric coefficients in quasi-coordinates, the $O(e^{-\eta't})$ error in the metric contributes only an $O(e^{-\eta't})$ error to the scalar curvature. Hence $s\left(\omega_{HS}+\sqrt{-1}\partial\bar\partial(H^1_{\ep,\tilde h,\tilde k}+\phi+\lambda t)\right)
\in\widetilde C^{0,\alpha}_{\eta,\delta-4}.$


The remaining terms in the definition of $Q_\ep$ lie in the same target space: $\mathcal D^*_{\omega_{HS}}\mathcal D_{\omega_{HS}}(H^1_{\ep,\tilde h,\tilde k})$ is supported away from the cusp end, $\ep^2\nu$ contributes only to the constant term, and $\mathcal D^*_{\omega_{HS}}\mathcal D_{\omega_{HS}}(\psi_f)\in C^{0,\alpha}_{\eta,\delta-4}$ by the construction of the linear inverse. Since $\nabla(H^1_{\ep,\tilde h,\tilde k}+\phi+\lambda t)$ is bounded, the gradient product term is also in
$C^{0,\alpha}_{\eta,\delta-4}$. This proves the claim.
    
From Theorem~\ref{theorem bounded inverse for cusp}, we know the operator 

    $$(\phi, f, \lambda)\mapsto\mathcal{D}^*_{\omega_{HS}}\mathcal{D}_{\omega_{HS}}(\phi+\psi_f+t\lambda)$$

    has a bounded right inverse, given our choice of weights in Theorem~\ref{theorem bounded inverse for cusp}. We consider an extension operator following \cite{AP09} Section 5.2. For given $R\ge 1$, consider $\mathcal E_R: C^{0, \al}_{\eta, \dt}(Bl_0 B_{R}\backslash E)\rightarrow C^{0, \al}_{\eta, \dt}(Bl_0\C^n\backslash E)$ defined by

    \begin{equation}\label{extension operator}
   \mathcal E_R(f(z)) = \begin{cases}
  f(z)  & \text{in } Bl_0\bar{B}_R\backslash E \\
  \bar{\chi}\left(\frac{|z|}{R}\right)f\left(R\frac{z}{|z|}\right) & \text{in } Bl_0B_{2R}\backslash Bl_0B_R\\
  0 & \text{in } Bl_0\C^n\backslash Bl_0B_{2R} 
\end{cases}
\end{equation}

    where $t\mapsto\bar{\chi}(t)$ is a smooth cut-off function which equals to $1$ for $t<\frac{5}{4}$ and vanishes for $t>\frac{7}{4}$. Let $\mathcal E_{\ep}$ be the extension operator for $R_{\ep}$, from \cite{AP09} we know it is a bounded operator whose norm is independent of $\ep$. With this operator, the Equation~\eqref{equation modified cscK} can be rewritten as a fixed point problem

    $$(\phi,f,\lambda)=N_{\ep}(\phi,f,\lambda).$$

    where we define
    
    \begin{equation}\label{N operator}
    \mathcal N_{\ep}: C^{4,\alpha}_{\eta,\delta}\times \ker\mathcal{D}^*_E\mathcal{D}_E\times\mathbb{R}
\rightarrow
C^{4,\alpha}_{\eta,\delta}\times \ker\mathcal{D}^*_E\mathcal{D}_E\times\mathbb{R}.
\end{equation}

\begin{lem}\label{Lemma L bound}
    For sufficiently small $c$, assume $\|\phi\|_{C^{4, \alpha}_{0, 2}}, |\lambda|\le c$, let 
    $$\omega_{HS}'\coloneqq\omega_{HS}+\displaystyle\sqrt{-1}\partial\bar{\partial}H^1_{\ep, \tilde{h}, \tilde{k}}, \quad \omega_{\lambda}\coloneqq\omega_{HS}'+\ddb\lambda t, \quad \omega_{\phi, \lambda}\coloneqq\omega_{HS}'+\displaystyle\sqrt{-1}\partial\bar{\partial}(\phi+\lambda t).$$
    
    Then there exists $C>0$, such that 

    $$\|L_{\omega_{\phi, \lambda}}-L_{\omega_{HS}'}\|_{C^{4, \alpha}_{\eta, \delta}\times\R\rightarrow \widetilde{C}^{0, \alpha}_{\eta, \delta-4}}\le C\|(\phi, \lambda)\|_{C^{4,\alpha}_{0, 2}\times\R}.$$
\end{lem}

\begin{proof}
    First we show there exists $C_1>0$ such that $\|g_{HS}'\|_{C^{2, \alpha}(Bl_0\C^n\backslash E)}\le C_1$. For this, we divide the regions into three parts, pulling back the metric components of $g_{HS}'$ to each region. Note that it is different from pulling back the metric tensors here, and no additional rescaling factor will be introduced. Near $E$, the metric is uniformly equivalent to the Poincar\'e hyperbolic metric; on the neck annulus, the metric is uniformly equivalent to the Euclidean flat metric; on $X\backslash B_1$, it's the cscK metric we started with. In each case, we have a uniform upper bound on the metric. Similarly, we get $\|g_{HS}'\|_{C^{2, \alpha}(Bl_0\C^n\backslash E)}^{-1}\le C_2$ for some $C_2>0$. Then we have

    \begin{multline}\label{L bound 1}
        \|L_{\omega_{\lambda}}(\varphi, \mu)-L_{\omega_{HS}'}(\varphi, \mu)\|_{\widetilde{C}^{0, \alpha}_{\eta, \delta-4}}=\left\Vert\left(g^{j\bar{k}}_{\lambda}g^{p\bar{q}}_{\lambda}-g'^{j\bar{k}}_{HS}g'^{p\bar{q}}_{HS}\right)\nabla_p\nabla_j\nabla_{\bar{k}}\nabla_{\bar{q}}(\varphi+\mu t)\right\Vert_{\widetilde{C}^{0, \alpha}_{\eta, \delta-4}}\\
        \le\left\Vert g^{j\bar{k}}_{\lambda}\left(g^{p\bar{q}}_{\lambda}-g'^{p\bar{q}}_{HS})-(g^{j\bar{k}}_{\lambda}-g'^{j\bar{k}}_{HS})g'^{p\bar{q}}_{HS}\right)\right\Vert_{C^{2, \alpha}}\|(\varphi, \mu)\|_{C^{4, \alpha}_{\eta, \delta}\times\R}.
    \end{multline}

    \begin{multline}\label{L bound 2}
        \|L_{\omega_{\phi, \lambda}}(\varphi, \mu)-L_{\omega_{\lambda}}(\varphi, \mu)\|_{\widetilde{C}^{0, \alpha}_{\eta, \delta-4}}=\left\Vert\left(g^{j\bar{k}}_{\phi, \lambda}g^{p\bar{q}}_{\phi, \lambda}-g^{j\bar{k}}_{\lambda}g^{p\bar{q}}_{\lambda}\right)\nabla_p\nabla_j\nabla_{\bar{k}}\nabla_{\bar{q}}(\varphi+\mu t)\right\Vert_{\widetilde{C}^{0, \alpha}_{\eta, \delta-4}}\\
        \le\left\Vert g^{j\bar{k}}_{\phi, \lambda}\left(g^{p\bar{q}}_{\phi, \lambda}-g^{p\bar{q}}_{\lambda})-(g^{j\bar{k}}_{\phi, \lambda}-g^{j\bar{k}}_{\lambda})g^{p\bar{q}}_{\lambda}\right)\right\Vert_{C^{2, \alpha}}\|(\varphi, \mu)\|_{C^{4, \alpha}_{\eta, \delta}\times\R}
    \end{multline}

We observe that with this inequality, now it suffices to give bounds on $g_{\phi, \lambda}^{-1}-g_{\lambda}^{-1}$ and $g_{\lambda}^{-1}-g_{HS}'^{-1}$. Note

$$g_{\phi, \lambda}^{-1}-g_{\lambda}^{-1}=g_{\phi, \lambda}^{-1}(g_{\lambda}-g_{\phi, \lambda})g_{\lambda}^{-1}, \quad g_{\lambda}^{-1}-g_{HS}'^{-1}=g_{\lambda}^{-1}(g_{HS}'-g_{\lambda})g_{HS}'^{-1},$$

then
        
\begin{equation}\label{lemma 8.2.2 estimate 1}
\left\Vert g_{\phi, \lambda}^{-1}-g_{\lambda}^{-1}\right\Vert_{C^{2, \alpha}}\le\left\Vert g_{\phi, \lambda}^{-1}\right\Vert_{C^{2, \alpha}}\left\Vert \sqrt{-1}\partial\bar{\partial}\phi\right\Vert_{C^{2, \alpha}}\left\Vert g_{\lambda}^{-1}\right\Vert_{C^{2, \alpha}}
\le C_3\|\phi\|_{C^{4, \alpha}_{0, 2}}.
\end{equation}

Note that $g_{HS}'$ is cscK metric near the divisor, we know from \cite{Auv14} Theorem 3.1(which we quote as Theorem \ref{theorem Auvray 3.1} in Section~\ref{section 9}) that 

 $$(g_{HS}')_{i\bar{j}}=\begin{pmatrix}
\displaystyle\frac{1}{|z_1|^2\log^2|z_1|^2} & 0\\
0 & g_{FS, i\bar{j}}
\end{pmatrix}+O(e^{-\dt t}).$$

Let $A:=\begin{pmatrix}
\displaystyle\frac{1}{|z_1|^2\log^2|z_1|^2} & 0\\
0 & g_{FS, i\bar{j}}
\end{pmatrix}$, then $A^{-1}=\begin{pmatrix}
|z_1|^2\log^2|z_1|^2 & 0\\
0 & g_{FS}^{i\bar{j}}
\end{pmatrix}$ which has bounded norm. Then for $(X)_{i\bar{j}}:=(g)_{i\bar{j}}-(A)_{i\bar{j}}$, we know $(A+X)^{-1}=A^{-1}-A^{-1}XA^{-1}+O(\|X\|^2)$. Thus

$$(g_{HS}')^{i\bar{j}}=A^{-1}+O(e^{-\dt t}).$$

From $(g_{HS}'-g_{\lambda})_{i\bar{j}}=\begin{pmatrix}
\displaystyle\frac{\lambda}{|z_1|^2\log^2|z_1|^2} & 0\\
0 & 0
\end{pmatrix}$, we see $\left\Vert(g_{HS}'-g_{\lambda})g_{HS}'^{-1}\right\Vert_{C^{2, \al}}\le C_4|\lambda|$, and thus

\begin{equation}\label{lemma 8.2.2 estimate 2}
\left\Vert g_{\lambda}^{-1}-g_{HS}'^{-1}\right\Vert_{C^{2, \alpha}}\le\left\Vert g_{\lambda}^{-1}\right\Vert_{C^{2, \alpha}}\left\Vert(g_{HS}'-g_{\lambda})g_{HS}'^{-1}\right\Vert_{C^{2, \al}}\le C_5|\lambda|.
\end{equation}

Substituting Equation~\eqref{lemma 8.2.2 estimate 1} back to Equation~\eqref{L bound 1} and Equation~\eqref{lemma 8.2.2 estimate 2} back to Equation~\eqref{L bound 2}, we get the desired result.
\end{proof}

\begin{lem}\label{nonlinear analysis}
    There exists a constant $c>0$ independent of $\Lambda$ and $c_{\Lambda}>0$, $\epsilon_{\Lambda}>0$ such that for $\epsilon\in(0, \ep_{\Lambda})$, 
    
    \begin{equation}\label{equation nonlinear 0 term}
         \|\mathcal{N}_{\ep}(0,0,0)\|_{C^{4, \alpha}_{\eta, \delta}\times \ker\mathcal{D}^*_E\mathcal{D}_E\times\mathbb{R}}\le c R_{\ep}^{3-2n-\dt}.
    \end{equation}
    
    If 
    
    $$\|(\phi_1, f_1, \lambda_1)\|_{C^{4,\alpha}_{\eta, \delta}\times \ker\mathcal{D}^*_E\mathcal{D}_E\times\mathbb{R}}, \quad\|(\phi_2, f_2, \lambda_2)\|_{C^{4,\alpha}_{\eta, \delta}\times \ker\mathcal{D}^*_E\mathcal{D}_E\times\mathbb{R}}\le 2c R_{\ep}^{3-2n-\dt},$$
    
    then

    \begin{equation}\label{equation nonlinear difference}
        \|\mathcal{N}_{\ep}(\phi_1, f_1, \lambda_1)-\mathcal{N}_{\ep}(\phi_2, f_2, \lambda_2)\|\le c_{\Lambda} R_{\ep}^{3-2n-\dt}\|(\phi_1, f_1, \lambda_1)-(\phi_2, f_2, \lambda_2)\|,
    \end{equation}

    and let $\tilde{\mathcal N}_{\ep}$ be the map taking into account the set of $\nu$ and the boundary data $\tilde{h}', \tilde{k}'$, then

    \begin{equation}\label{equation nonlinear difference with boundary}
        \|\tilde{\mathcal N}_{\ep}(\phi, f, \lambda, \nu, \tilde{h}, \tilde{k})-\tilde{\mathcal N}_{\ep}(\phi, f, \lambda, \nu', \tilde{h}', \tilde{k}')\|\le c_{\Lambda}(R_{\ep}^{-1}\|(\tilde{h}-\tilde{h}', \tilde{k}-\tilde{k}')\|_{C^{4, \alpha}\times C^{2, \al}}+R_{\ep}^{3-2n-\dt}|\nu-\nu'|).
    \end{equation}
    
\end{lem}

\begin{proof}

For Equation~\eqref{equation nonlinear 0 term}, note from \ref{choice of H^1}, $H^1_{\ep, \tilde{h}, \tilde{k}}$ is supported away from the cusp end, we can argue similarly as in \cite{AP09} Lemma 5.3. More precisely, we have

$$\|\nabla^2H^1_{\ep, \tilde{h}, \tilde{k}}\|_{C^{0, 2}(Bl_0B_2\backslash E)}\le c_{\Lambda} R_{\ep}^{2-2n},$$

then note $L_{\omega_{HS}}H^1_{h,k}=(L_{\omega_{HS}}-\Delta_{Euc}^2)H^1_{h,k}$ on the annulus $Bl_0 B_{R_{\ep}}\backslash Bl_0B_{2}$, we see from the expansions \ref{asymptotics of HS n=2}, \ref{asymptotics of HS n ge 3}, the coefficients of $\nabla^{2+i}$ live in $C^{0, \al}_{i-2n}$ for $i=0, 1, 2$ and thus

$$\|L_{\omega_{HS}}H^1_{\ep, \tilde{h}, \tilde{k}}\|_{\widetilde{C}^{0, \al}_{\eta, \dt-4}}=\|L_{\omega_{HS}}H^1_{\ep, \tilde{h}, \tilde{k}}\|_{C^{0, \al}_{\dt-4}}\le c_{\Lambda}R_{\ep}^{2-2n}R_{\ep}^{-2n}R_{\ep}^{4-\dt}=c_{\Lambda}R_{\ep}^{2-2n},$$

and

$$\|\mathcal E(Q_{\omega_{HS}}(\nabla^2 H^1_{\ep, \tilde{h}, \tilde{k}}))\|_{\widetilde{C}^{0, \al}_{\eta, \dt-4}}=\|\mathcal E(Q_{\omega_{HS}}(\nabla^2 H^1_{\ep, \tilde{h}, \tilde{k}}))\|_{C^{0, \al}_{\dt-4}}\le c_{\Lambda}R_{\ep}^{4-4n}R_{\ep}^{-4n}R_{\ep}^{4-\dt}= R_{\ep}^{4-4n}.$$

Also 

\begin{equation}\label{estimate in Lemma 8.2.3}
    \|\ep^2\nu\|_{\widetilde{C}^{0, \alpha}_{\eta, \delta-4}}=\|\ep^2\nu\|_{C^{0, \al}_{\dt-4}}\le cR_{\ep}^{4-\dt}\ep^2\le R_{\ep}^{3-2n-\dt}.
\end{equation}

Combining these inequalities with the bound for $G_{\ep}^{-1}$, we obtain Equation~\eqref{equation nonlinear 0 term}.

For Equation~\eqref{equation nonlinear difference}, note for our choice of weights, on $Bl_0 B_{R_{\ep}}\backslash E$,

$$\|\phi\|_{C^{4, \al}_{0, 2}}\le c\|\phi\|_{C^{4, \al}_{\eta, \dt}},$$

then as in \cite{Sze14} Lemma 8.18 and Lemma 8.21, it follows from Lemma~\ref{Lemma L bound} by the mean value theorem that the $Q$ term is bounded. It remains to estimate the following

\begin{equation}\label{bound of additional term}
   \frac{1}{2}\|\langle \nabla\mathcal D^*_{\omega_{HS}}\mathcal D_{\omega_{HS}}(\psi_{f_2}), \nabla(H^1_{\ep, \tilde{h}, \tilde{k}}+\phi_2+\lambda_2t)\rangle-\langle \nabla\mathcal D^*_{\omega_{HS}}\mathcal D_{\omega_{HS}}(\psi_{f_1}), \nabla(H^1_{\ep, \tilde{h}, \tilde{k}}+\phi_1+\lambda_1t)\rangle\|_{\widetilde{C}^{0, \al}_{\eta, \dt-4}}. 
\end{equation}

Note $\|\nabla H^1_{\ep, \tilde{h}, \tilde{k}}\|_{C^{0, 2}}\le c_{\Lambda}R_{\ep}^{3-2n}$, $\|\nabla\mathcal D^*_{\omega_{HS}}\mathcal D_{\omega_{HS}}(\psi_{f_2}-\psi_{f_1})\|_{C^{0, \al}}\le c\|f_1-f_2\|$, then

$$\|\langle \nabla\mathcal D^*_{\omega_{HS}}\mathcal D_{\omega_{HS}}(\psi_{f_2}-\psi_{f_1}), \nabla H^1_{\ep, \tilde{h}, \tilde{k}}\rangle\|\le c_{\Lambda}R_{\ep}^{3-2n-\dt}\|f_1-f_2\|,$$

and for the remaining part of \ref{bound of additional term}, we have

\begin{equation*}
\begin{aligned}
 &\|\langle \nabla\mathcal D^*_{\omega_{HS}}\mathcal D_{\omega_{HS}}\psi_{f_2}, \nabla(\phi_2+\lambda_2t)\rangle-\langle \nabla\mathcal D^*_{\omega_{HS}}\mathcal D_{\omega_{HS}}\psi_{f_1}, \nabla(\phi_1+\lambda_1t)\rangle\|_{\widetilde{C}^{0, \al}_{\eta, \dt-4}}\\
&\le\|\langle \nabla\mathcal D^*_{\omega_{HS}}\mathcal D_{\omega_{HS}}\psi_{f_2}, \nabla(\phi_2+\lambda_2t)-\nabla(\phi_1+\lambda_1t)\rangle\|_{\widetilde{C}^{0, \al}_{\eta, \dt-4}}+\|\langle \nabla\mathcal D^*_{\omega_{HS}}\mathcal D_{\omega_{HS}}(\psi_{f_2}-\psi_{f_1}), \nabla(\phi_1+\lambda_1t)\rangle\|_{\widetilde{C}^{0, \al}_{\eta, \dt-4}}\\
&\le\|\nabla\mathcal D^*_{\omega_{HS}}\mathcal D_{\omega_{HS}}\psi_{f_2}\|_{C^{0, \al}_{0, -3}}\|(\phi_2, \lambda_2)-(\phi_1, \lambda_1)\|+\|(\phi_1, \lambda_1)\|\|\nabla\mathcal D^*_{\omega_{HS}}\mathcal D_{\omega_{HS}}(\psi_{f_2}-\psi_{f_1})\|_{C^{0, \al}_{0, -3}}\\
&\le c(\|f_2\|\cdot\|(\phi_2, \lambda_2)-(\phi_1, \lambda_1)\|+\|(\phi_1, \lambda_1)\|\cdot\|f_2-f_1\|)
\end{aligned}    
\end{equation*}

These inequalities combined with the bounds of $G^{-1}_{\ep}$ and $\mathcal E_{\ep}$ give the desired estimate.

For Equation~\eqref{equation nonlinear difference with boundary}, again since $H^1_{\ep, \tilde{h}, \tilde{k}}$ is supported away from the cusp end, we have from \cite{AP09} Lemma 5.3 the following estimates  

$$\|L_{\omega_{HS}}(H^1_{\ep, \tilde{h}, \tilde{k}}-H^1_{\ep, \tilde{h}', \tilde{k}'})\|_{C^{0, \al}_{\dt-4}}\le c_{\Lambda}R_{\ep}^{-1}\|(\tilde{h}-\tilde{h}', \tilde{k}-\tilde{k}')\|, $$

and 

$$\|\mathcal E (Q_{\omega_{HS}}(\nabla^2H^1_{\ep, \tilde{h}, \tilde{k}}-\nabla^2H^1_{\ep, \tilde{h}', \tilde{k}'})\|_{C^{0, \al}_{\dt-4}}\le c_{\Lambda}R_{\ep}^{-1}\|(\tilde{h}-\tilde{h}', \tilde{k}-\tilde{k}')\|.$$

The estimate for the variation in $\nu, \nu'$ follows from Equation~ \eqref{estimate in Lemma 8.2.3}.

\end{proof}

Note by reducing $c_{\Lambda}$ if necessary, we have $c_{\Lambda}R_{\ep}^{3-2n-\dt}\le\displaystyle\frac{1}{2},$ then with the above estimates, we know the map $N_{\ep}$ is a contraction on the set
    
    $$\mathcal{U}_{\ep}:\{(\phi, f, \lambda)\in C^{4, \alpha}_{\eta, \delta}\times \ker\mathcal{D}^*_E\mathcal{D}_E\times\mathbb{R}: \|(\phi, f, \lambda)\|_{C^{4, \alpha}_{\eta, \delta}\times\ker\mathcal{D}^*_E\mathcal{D}_E\times\mathbb{R}}\le cR_{\ep}^{3-2n-\dt}\}, $$

    for $\epsilon$ sufficiently small, then $\mathcal{N}_{\ep}$ is a contraction on $\mathcal{U}$ and $\mathcal{N}_{\ep}(\mathcal{U})\subset \mathcal{U}$. In particular, $\mathcal{N}_{\ep}$ has a fixed point, finishing the proof of Theorem~\ref{theorem modified cscK}.

\end{proof}

With these metrics on each piece, upon applying Theorem~\ref{theorem modified cscK} to

$$\nu\coloneqq s(\omega_X)+\nu_{\ep, h,k},$$

we define on $Bl_0B_1\backslash Bl_0B_{\frac{1}{2}}$

$$\psi^1\coloneqq \ep^2(\varphi_{HS}+H^1_{\ep, \tilde{h}, \tilde{k}}+\phi_{\ep}).$$

From Theorem~\ref{theorem cscK on punctured base}, we define on $B_2\backslash B_1$

$$\psi^0\coloneqq \varphi_X+H^0_{\ep, h, k}+\varphi_{\ep, h, k}.$$

To match up these two metrics, we note for our choice of $r_{\ep}$ in \ref{choice of r}, $\|\psi^0-H^0_{\ep, h, k}\|_{C^{4, \al}}\le cr_{\ep}^4$ and $\|\psi^1-\ep^2H^1_{\ep, \tilde{h}, \tilde{k}}\|_{C^{4, \al}}\le cR_{\ep}^{3-2n}$. Combined with the estimates \ref{nonlinear estimate for base 01}, \ref{nonlinear estimate for base 02} and \ref{equation n ge 3 nonlinear estimate}, we have the following.

\begin{thm}(\cite{AP09}, Section 5.3)\label{theorem gluing metrics}
    Consider functions $\psi^0$ and $\psi^1$ defined as above. There exists boundary data $h, k$ satisfying \ref{conditions for h,k} and $\tilde{h}, \tilde{k}$ satisfying \ref{conditions for tilde h,k}, such that on $\partial B_1$,

    $$\psi^0=\psi^1, \quad \partial_r\psi^0=\partial_r\psi^1, \quad \Delta\psi^0=\Delta\psi^1, \quad \partial_r\Delta\psi^0=\partial_r\Delta\psi^1,$$

    where $r=|x|$, with $x$ being the coordinates in $B_2$. Then we can glue along $\partial B_1$ the functions $\psi^0$ and $\psi^1$ and obtain a function on $Bl_pX\backslash E$, which results in a metric $\omega_{sol}$ on $Bl_pX\backslash E$.
\end{thm}

\begin{proof}[Proof of Theorem~\ref{Main theorem}]

Noticing Equation~\eqref{equation n ge 3 nonlinear estimate}, the existence of such boundary data can be obtained with the same arguments as in \cite{AP09}, Section 5.3. Near $E$, we know $\phi+\lambda t=O(t)$ from the definition, which is indeed of Poincar\'e type.


\medskip

We want to show that $\omega_{sol}$ is indeed cscK by arguing that the component $\mathcal D^*_{\omega_{HS}}\mathcal D_{\omega_{HS}}(\psi_f)$ is zero. Note the metric $g_{HS}$ near $E$ has expansion

$$g_{HS}=\displaystyle\frac{1}{n(n-1)}(dt^2+e^{-2t}\Theta^2)+p^*\omega_{FS}+O(e^{-t}),$$

then near $E$, $\omega_{sol}$ has expansion 

\begin{equation}\label{expansion of solution near E}
    g_{sol}=\ep^2\left(\displaystyle\frac{1}{n(n-1)}-\frac{1}{2}\lambda_{\ep}\right)(dt^2+e^{-2t}\Theta^2)+p^*\omega_{FS}+O(e^{-\eta' t})
\end{equation}

for some $\eta'>0$. Hence $\nabla^{1,0}s(\omega_{sol})$, when restricted to $E$, becomes the extremal vector field $\nabla^{1,0}s(\omega_{FS})$ on $E$, which equals to zero. On the other hand, recall that 

$$s(\omega_{sol})=\ep^2\nu+\mathcal D^*_{\omega_{HS}}\mathcal D_{\omega_{HS}}(\psi_{f_{\ep}})+\frac{1}{2}\langle \nabla\mathcal D^*_{\omega_{HS}}\mathcal D_{\omega_{HS}}(\psi_{f_{\ep}}), \nabla(H^1_{\ep, \tilde{h}, \tilde{k}}+\phi_{\ep}+\lambda_{\ep} t)\rangle,$$

then restricting the associated vector field to $E$, the parts involving $\mathcal D^*_{\omega_{HS}}\mathcal D_{\omega_{HS}}(\psi_{f_{\ep}})$ component gives nonzero term, and $\nabla^{1,0}(\ep^2\nu)$ gives zero. Hence, $\mathcal D^*_{\omega_{HS}}\mathcal D_{\omega_{HS}}(\psi_{f_{\ep}})$ is forced to vanish, and the metric is indeed cscK. The smoothness of $\phi$ follows from standard elliptic bootstrap arguments. 

\end{proof}

\section{Applications and Discussions}\label{section 9}

\subsection{Proof of Corollary \ref{cor SFK}}\label{subsection 9.1}

Note near $E$, we have an expansion of $g_{sol}$ as in Equation~\eqref{expansion of solution near E}. In fact, we have a more precise relation of the coefficient before $t$ in the expression of potential with the average scalar curvature. Recall from \cite{Auv14} the following:

\begin{thm}(Auvray, \cite{Auv14}, Theorem 3.1)\label{theorem Auvray 3.1}

    Given $D$ a smooth divisor and $\omega_{\varphi}$ a cscK Poincar\'e type K\"ahler metric on $Y\backslash D$ in the class $[\omega_Y]$, we have the asymptotic expression of metric $g_{\varphi}$
    $$g_{\varphi}=a(dt^2+e^{-2t}\Theta^2)+\pi^*g_D+O(e^{\eta_0t}),$$

    for some constants $\eta_0<0$ and $a>0$, more precisely, let 
    $$\bar{s}(\omega_Y)\coloneqq n\displaystyle\frac{-c_1(K_Y[D])\cup[\omega_Y]^{n-1}}{[\omega_Y]^n},\quad \bar{s}(\omega_Y|_D)\coloneqq (n-1)\displaystyle\frac{-c_1(K_Y[D]|_D)\cup[\omega_Y|_D]^{n-2}}{[\omega_Y|_D]^{n-1}},$$

    we have an explicit expression of $a$, given by

    $$a=\displaystyle\frac{1}{\bar{s}(\omega_Y|_D)-\bar{s}(\omega_Y)}.$$
\end{thm}

Applying this theorem to $Y=Bl_pX$, $D=E$, we have 

$$\frac{1}{\epsilon^2(\frac{1}{n(n-1)}-\frac{\lambda_{\ep}}{2})}=\bar{s}(\omega_{\epsilon}|_E)-\bar{s}(\omega_{sol}).$$ 

We are interested in how fast the RHS grows; for this purpose, we calculate the scalar curvature of the new cscK metric. Note $\omega_{\epsilon}$ and our solution $\omega_{sol}$ to the cscK equation are in the same K\"ahler class, which is $\pi^*[\omega_X]-\epsilon^2[E]$, at the same time, we have

$$c_1(Bl_pX)=\pi^*c_1(X)-(n-1)[E].$$

Recall that $[E]$ denotes the Poincar\'e dual of $E$ and below by $c_1([E])$ we mean $c_1(\mathcal{O}(E))$, we have

$$c_1(Bl_pX)-c_1([E])=\pi^*c_1(X)-(n-1)[E]-[E]=\pi^*c_1(X)-n[E].$$

Recall from \cite{GH14} that $[E]^{n}=(-1)^{n-1}$, we get 

$$(c_1(Bl_pX)-c_1([E]))\cup[\omega_{\epsilon}]^{n-1}=c_1(X)\cup[\omega_X]^{n-1}-n\epsilon^{2n-2},\quad [\omega_{\epsilon}]^n=[\omega_X]^n-\epsilon^{2n},$$

thus 

\begin{equation}\label{expansion of scalar curvature of solution}
   \bar{s}(\omega_{sol})=\displaystyle\frac{n(c_1(X)\cup[\omega_X]^{n-1}-n\epsilon^{2n-2})}{[\omega_X]^n-\epsilon^{2n}}, 
\end{equation}

which is a decreasing function for small $\epsilon$ and  converges to $s(\omega_X)$ as $\epsilon\rightarrow0$. 

\medskip

We also compute the average scalar curvature of metric on $E$ in the class $[\omega_{\ep}|_E]$,

$$c_1(K_{Bl_pX}[E])\cup[\omega_{\epsilon}|_E]^{n-2}=(\pi^*c_1(X)-n[E])\cdot(\pi^*[\omega_X]-\epsilon^2[E])^{n-2}\cdot[E]=\epsilon^{2n-4}n,$$

$$[\omega_{\epsilon}|_E]^{n-1}=[\omega_{\epsilon}]^{n-1}\cdot c_1([E])=\epsilon^{2n-2},$$

thus 

$$\bar{s}(\omega_{\epsilon}|_E)=\displaystyle\frac{n(n-1)}{\epsilon^{2}}.$$

\begin{proof}[Proof of Corollary \ref{cor SFK}]
    As in the proof of \cite{arezzo2006blowing} Corollary 8.1, we can apply LeBrun-Simanca and obtain a family of cscK metrics $\omega(y)$ such that $\omega(0)=\omega_X$ and $s(\omega(y))<0$ for $y\in[-y_0, 0)$ and $s(\omega(y))>0$ for $y\in(0, y_0]$. Apply Theorem~\ref{Main theorem} to each $\omega(y)$, we obtain a family $\omega_{\epsilon}(y)$ of cscK metrics with a uniform bound for $\epsilon$ as $y$ varies in $[-y_0, y_0]$, denoted by $\ep_0$. Arguing similarly as in \cite{arezzo2006blowing}, we know $\omega_{\epsilon}(y)$ depends continuously on $y$. From the discussion at the beginning of this section, if $\epsilon_0$ is small enough, we know $s(\omega_{\epsilon}(y))<0$ if $y<0$ and $s(\omega_{\epsilon}(y))>0$ if $y>0$. Now that $s(\omega_{\epsilon}(y))$ depends continuously on $y$, there exists $y'$ such that $s(\omega_{\epsilon}(y'))=0$.
\end{proof}

\subsection{Extension of results to more general situations}\label{subsection 9.2}









In this subsection, we extend the arguments to a cscK manifold $X$ which admits non-trivial holomorphic vector fields. We will show that if we choose the blow-up points carefully, the gluing construction we have still work under some modifications. Now to construct cscK metric on $X_{r_{\ep}}$ piece, we define function 

$$G(z)\coloneqq \begin{cases}
    -\log|z| &\text{ when }n=2\\
    |z|^{4-2n} &\text{ when }n\ge3
\end{cases}$$

which are homogeneous solutions to $L_{\omega_{Euc}}$. On the gluing annulus, these can be perturbed to $\tilde{G}$, which are homogeneous solutions to $L_{\omega_X}$ on $B_{r_0}^*$ for sufficiently small $r_0$. We define the deficiency space as in \cite{AP09} Section 4.1

$$D_{\delta}(L_{\omega_X})\coloneqq\text{Span }\{\chi_i\tilde{G_i}, i=1, \cdots, k\}.$$

With this, given the genericity condition on the points, we obtain a surjective operator $   C^{4,\al}_{\dt}(X_p)\times D_{\delta}(L_{\omega_X})\rightarrow C^{0,\al}_{\dt-4}(X_p)$ given by $(\phi, \beta)\mapsto L_{\omega_X}(\phi+\beta).$ For the first perturbation on the gluing annulus, we require an additional term derived from the deficiency space to deal with the complexity introduced by $\mathfrak{k}$ as in \cite{AP09}. The next step is to find the cscK metric on both pieces. Choose $r_{\ep}$ and $R_{\ep}$ as before, on $X_{r_{\ep}}$, the existence of cscK metric follows from \cite{AP09} Proposition 5.1. On $Bl_0 B_{R_{\ep}}\backslash E$, we perturb $a^2\omega_{HS}$ for a given constant $a>0$ and we require $a\in[a_{\min}, a_{\max}]$ for some fixed constants $0<a_{\min}<a_{\max}$. Then, with essentially similar estimates as in Section \ref{section 8}, we solve a modified cscK equation. 

\begin{thm} \label{AP09 theorem modified cscK}
   For a constant $\Lambda>0$ large enough, there exists $c>0$ independent of $\Lambda$ and $\epsilon_{\Lambda}>0$ such that $\forall\ep\in(0, \ep_{\Lambda})$, and for all $\nu\in\R$ and boundary data $\tilde{h}\in C^{4, \al}(\partial B_1), \tilde{k}\in C^{2, \al}(\partial B_1)$ satisfying \ref{conditions for tilde h,k}, and $a\in[a_{\min}, a_{\max}]$, $\nu\in[\nu_{\min}, \nu_{\max}]$, then there exists $\phi_{\ep}\in C^{4, \al}_{\eta, \dt}$, constant $\lambda_{\ep}\in\R$ and $f_{\ep}\in\ker\mathcal D^*_E\mathcal D_E$ such that the scalar curvature of the K\"ahler form
   
   $$\omega_{\epsilon, \tilde{h},\tilde{k}, HS}=a^2\omega_{HS}+\displaystyle\sqrt{-1}\partial\bar{\partial}(H^1_{\ep, \tilde{h}, \tilde{k}}+a^{2n-2}R_{\ep}^{4-2n}G+\phi_{\ep}+\lambda_{\ep}t)$$
   
   defined on $Bl_0 B_{R_{\ep}}\backslash E$ is
   
   $$\ep^2\nu+\mathcal D^*_{\omega_{HS}}\mathcal D_{\omega_{HS}}(\psi_{f_{\ep}})+\frac{1}{2}\langle \nabla\mathcal D^*_{\omega_{HS}}\mathcal D_{\omega_{HS}}(\psi_{f_{\ep}}), \nabla(H^1_{\ep, \tilde{h}, \tilde{k}}+\phi_{\ep}+\lambda_{\ep} t)\rangle.$$
   
   Further,

   $$\|\phi_{\ep}(R_{\ep}\cdot/a)\|_{C^{4, \al}(B_2\backslash B_1)}\le cR_{\ep}^{3-2n}, \quad |\lambda_{\ep}|+\|f_{\ep}\|\le cR_{\ep}^{3-2n-\dt}.$$

   If $\phi_{\ep}, \lambda_{\ep}$ are determined by boundary data $\tilde{h}, \tilde{k}$ and $\phi'_{\ep}, \lambda'_{\ep}$ are determined by boundary data $\tilde{h}', \tilde{k}'$, then

   \begin{equation}\label{AP09 equation n ge 3 nonlinear estimate}
       \|f_{\ep}-f_{\ep}'\|+|\lambda_{\ep}-\lambda'_{\ep}|+\|\phi(R_{\ep}\cdot/a)-\phi'(R_{\ep}\cdot/a')\|\le c_{\Lambda}(R_{\ep}^{\dt-1}\|(\tilde{h}-\tilde{h}',\tilde{k}-\tilde{k}')\|+R_{\ep}^{3-2n}(|\nu-\nu'|+|a-a'|)).
   \end{equation}
\end{thm}

The matching of the metrics constructed on the two pieces is independent of the Poincar\'e type behavior near the exceptional divisor, and the matching arguments in \cite{AP09} Section 5.3 still hold in our case. Finally, because of the condition (\ref{theorem condition}), the metric on $Bl_pX\backslash E$ is cscK. We postpone the details of this argument to the more complicated situation of extremal  K\"ahler metrics at the end of the following subsection.

\subsection{Extremal K\"ahler metric}\label{subsection 9.3}

We discuss the situation of extremal K\"ahler(extK) metrics in Theorem~\ref{thm application 2}. To solve the extK equation, we need to choose $T$-invariant coordinates when defining weighted H\"older spaces, and we want to solve for $T$-invariant potential functions. Comparing the expansion of $\omega_{HS}$ in \ref{asymptotics of HS n=2} and \ref{asymptotics of HS n ge 3} with those of Burns--Simanca metrics, we see the arguments of perturbations for $\omega_X$ to an extK metric that on \cite{APS11} Proposition 6.1.1. Again, the difference lies in the piece of $Bl_0 B_{R_{\ep}}\backslash E$, where we solve a modified extK equation. 

\begin{thm}\label{theorem modified extK}
    Given $\nu\in\R, X\in\mathfrak t$ satisfying 
    
    $$|\ep^2\nu|+\|X\|\le c,$$

    we assume that the lift of $X$ to $Bl_0\C^n$ vanishes along the exceptional divisor,

    \begin{equation}\label{local version of condition 3}
        X|_E=0.
    \end{equation}

    There exists $\theta>0, c>0$ and for all $\Lambda>0$, there exists $\ep_{\Lambda}>0$ such that for all $\ep\in(0, \ep_{\Lambda})$, $a\in[a_{\min}, a_{\max}]$, and $T$-invariant boundary data satisfying

    $$\|\tilde{h}\|_{C^{4, \al}(\partial B_1)}+\|\tilde{k}\|_{C^{2, \al}(\partial B_1)}\le \Lambda R_{\ep}^{3-2n}, \quad \displaystyle\int_{\partial B_1}(4n\tilde{h}-\tilde{k})dvol=0,$$
    
    there exists a $T$-invariant function $\phi_{\ep}$, a constant $\lambda_{\ep}$ and $f_{\ep}\in\ker\mathcal D^*_E\mathcal D_E$ such that the scalar curvature of the K\"ahler form
   
   $$\omega_{\epsilon, \tilde{h},\tilde{k}, HS}=a^2\omega_{HS}+\displaystyle\sqrt{-1}\partial\bar{\partial}(H^1_{\ep, \tilde{h}, \tilde{k}}+\phi_{\ep}+\lambda_{\ep}t)$$
   
   defined on $Bl_0 B_{R_{\ep}}\backslash E$ is
   
   $$Y_{\ep}+\frac{1}{2}\langle\nabla Y_{\ep}, \nabla(H^1_{\ep, \tilde{h}, \tilde{k}}+\phi_{\ep}+\lambda_{\ep}t)\rangle,$$
   
   where $Y_{\ep}=\ep^4Y+a^{-4}\mathcal D^*_{\omega_{HS}}\mathcal D_{\omega_{HS}}(\psi_{f_{\ep}})$, with $Y$ being the potential function for $X$, and
   
   $$\displaystyle\frac{1}{|\partial B_1|}\displaystyle\int_{\partial B_1}s(\omega_{\epsilon, \tilde{h},\tilde{k}, HS}(R_{\ep}\cdot/a))dvol=\ep^2\nu.$$
   
   Further,

   $$\|\phi_{\ep}(R_{\ep}\cdot/a)\|_{C^{4, \al}(B_2\backslash B_1)}\le cR_{\ep}^{3-2n}, \quad |\lambda_{\ep}|+\|f_{\ep}\|\le cR_{\ep}^{3-2n-\dt}.$$

   If $\phi_{\ep}, \lambda_{\ep}$ are determined by boundary data $\tilde{h}, \tilde{k}$ and $\phi'_{\ep}, \lambda'_{\ep}$ are determined by boundary data $\tilde{h}', \tilde{k}'$, then

   \begin{equation}\label{APS equation n ge 3 nonlinear estimate}
       \|f_{\ep}-f_{\ep}'\|+|\lambda_{\ep}-\lambda'_{\ep}|+\|\phi(R_{\ep}\cdot/a)-\phi'(R_{\ep}\cdot/a')\|\le c_{\Lambda}(R_{\ep}^{\dt-1}\|(\tilde{h}-\tilde{h}',\tilde{k}-\tilde{k}')\|+R_{\ep}^{3-2n}(|\nu-\nu'|+|a-a'|)+R_{\ep}^{4-4n}\|Y-Y'\|).
   \end{equation}
\end{thm}

Note that Condition~\ref{local version of condition 3} is the local version of Condition~\ref{theorem condition} in Theorem~\ref{thm application 2}. Assume we already have this theorem, we know from standard arguments in \cite{APS11} that on $X_{r_{\ep}}$, there exists extK metric $\omega_{\ep, h,k}$ with extremal vector field $X_{\ep, h, k, a}$, then apply Theorem~\ref{theorem modified extK} to

$$\nu=\displaystyle\frac{1}{|\partial B_1|}\displaystyle\int_{\partial B_1}s(\omega_{\ep, h,k})dvol, \quad X=X_{\ep, h, k, a}.$$

With these data, we match the metric on $Bl_0 B_{R_{\ep}}\backslash E$ with the extK metric on $X_{r_{\ep}}$. To prove this theorem, the modified extK equation we want to solve is

\begin{equation}\label{modified extK equation}  s(a^2\omega_{HS}+\displaystyle\sqrt{-1}\partial\bar{\partial}(H^1_{\ep, \tilde{h}, \tilde{k}}+\phi_{\ep}+\lambda_{\ep}t))=Y_{\ep}+\frac{1}{2}\langle\nabla Y_{\ep}, \nabla(H^1_{\ep, \tilde{h}, \tilde{k}}+\phi_{\ep}+\lambda_{\ep}t)\rangle.
\end{equation}

Expanding the scalar curvature operator into its linear $\mathcal D^*\mathcal D$ term and nonlinear $Q$ term, we rewrite the equation as

$$\begin{aligned}
a^{-4}\mathcal D^*\mathcal D(\phi_{\ep}+\lambda_{\ep}t+\psi_{f_{\ep}})
=&Q(H^1_{\ep, \tilde{h},\tilde{k}}+\phi_{\ep}+\lambda_{\ep}t)-\mathcal D^*\mathcal D(H^1_{\ep, \tilde{h}, \tilde{k}})-\ep^4Y\\
&-\frac{1}{2}\langle\nabla(\ep^4Y+a^{-4}\mathcal D^*\mathcal D\psi_{f_{\ep}}),\nabla(H^1_{\ep, \tilde{h}, \tilde{k}}+\phi_{\ep}+\lambda_{\ep}t)\rangle.
\end{aligned}$$

and we denote the RHS by $Q_{\ep}$. We claim that $Q_{\ep}\in \widetilde C^{0,\alpha}_{\eta,\delta-4}$. The scalar curvature part is treated exactly as in Theorem~\ref{theorem modified cscK} and lives in the modified H\"older space. It remains only to check the terms involving $Y$. By Condition~\ref{local version of condition 3}, the lifted vector field satisfies $X|_E=0$. Since $Y$ is the potential of $X$, then along $E$, $\nabla^{1,0}_{\omega_{FS}}(Y|_E)=X|_E=0$. It follows that $Y\in\widetilde C^{0,\alpha}_{\eta,\delta-4}$.

Since $X|_E=0$, the gradient $\nabla Y$ vanishes along $E$ and therefore decays in the cusp direction. Again we have $\nabla(H^1_{\ep,\tilde h,\tilde k}+\phi+\lambda t)$ is bounded, then 

$$\left\langle \nabla Y,\nabla(H^1_{\ep,\tilde h,\tilde k}+\phi+\lambda t)\right\rangle\in C^{0,\alpha}_{\eta,\delta-4}.$$

The terms involving $\mathcal D^*\mathcal D\psi_{f_\ep}$ are handled as in Theorem~\ref{theorem modified cscK}. Hence $Q_\ep\in \widetilde C^{0,\alpha}_{\eta,\delta-4}$.

Again, we have a bounded inverse of the LHS and an extension operator whose norm is independent of $\ep$. Thus it suffices to solve the fixed point problem for $N_{\ep}=G_{\ep}^{-1}\circ\mathcal E_{\ep}\circ Q_{\ep}$.

\begin{lem} \label{lemma nonlinear in APS}
    There is a constant $c>0$ independent of $\Lambda$ and $c_{\Lambda}>0$, $\epsilon_{\Lambda}>0$ such that for $\epsilon\in(0, \ep_{\Lambda})$, 
    
    \begin{equation}\label{equation nonlinear 0 term in APS}
         \|\mathcal{N}_{\ep}(0,0,0)\|_{C^{4, \alpha}_{\eta, \delta}\times \ker\mathcal{D}^*_E\mathcal{D}_E\times\mathbb{R}}\le c R_{\ep}^{3-2n-\dt}.
    \end{equation}

    If 
    
    $$\|(\phi_1, f_1, \lambda_1)\|_{C^{4,\alpha}_{\eta, \delta}\times \ker\mathcal{D}^*_E\mathcal{D}_E\times\mathbb{R}}, \quad\|(\phi_2, f_2, \lambda_2)\|_{C^{4,\alpha}_{\eta, \dt}\times \ker\mathcal{D}^*_E\mathcal{D}_E\times\mathbb{R}}\le 2c R_{\ep}^{3-2n-\dt},$$
    
    then

    \begin{equation}\label{equation nonlinear difference in APS}
        \|\mathcal{N}_{\ep}(\phi_1, f_1, \lambda_1)-\mathcal{N}_{\ep}(\phi_2, f_2, \lambda_2)\|\le c_{\Lambda}(R_{\ep}^{3-2n-\dt}+\ep^4R_{\ep}^4)\|(\phi_1, f_1, \lambda_1)-(\phi_2, f_2, \lambda_2)\|,
    \end{equation}

    and let $\tilde{\mathcal N}_{\ep}$ be the map taking into account the divisor volume $a$ and the set of boundary data $\tilde{h}', \tilde{k}'$, then

    \begin{equation}\label{equation nonlinear difference with boundary in APS}
        \|\tilde{\mathcal N}_{\ep}(\phi, f, \lambda, \tilde{h}, \tilde{k})-\tilde{\mathcal N}_{\ep}(\phi, f, \lambda, \tilde{h}', \tilde{k}')\|\le c_{\Lambda}(R_{\ep}^{-1}\|(\tilde{h}-\tilde{h}', \tilde{k}-\tilde{k}')\|+R_{\ep}^{3-2n-\dt}|\nu-\nu'|+R_{\ep}^{4-4n-\dt}\|Y-Y'\|).
    \end{equation}
    
\end{lem}

\begin{proof}
    The proof follows mostly from \cite{APS11} Proposition 6.2.1. Only note that near the exceptional divisor, we again use Lemma~\ref{Lemma L bound} and mean value theorem.
\end{proof}

We use \cite{APS11} Lemma 7.0.2 to match the metrics on the two pieces and obtain a metric $\omega_{sol}$ on $Bl_pX\backslash E$. To show the $\mathcal D^*\mathcal D(\psi_{f_{\ep}})$ term vanishes, we use the condition (\ref{theorem condition}). To make sense of the condition, we first discuss the definition of the extremal vector field. 

\begin{definition}
    For a compact manifold $X$ with smooth divisor $E$, we say a holomorphic vector field $Z$ is \textit{parallel to} $E$ if their normal components vanish along $E$, i.e., if $E=\{z_1=0\}$, and locally we write $Z=\Re\left(\displaystyle\sum_{i=1}^mf_i\displaystyle\frac{\partial}{\partial z_i}\right)$, then $f_1|_E=0$. We denote the space of holomorphic potentials parallel to $E$ as $h_{E}^{\slash\slash}$.
\end{definition} 

A key property of such vector fields is that they are $L^2$ with respect to Poincar\'e type K\"ahler metric on $X\backslash E$. For $\omega$ a Poincar\'e type K\"ahler metric in $[\omega_X]$, Auvray in \cite{Auv13-2} defined Poincar\'e type Futaki character of $Z\in h_{E}^{\slash\slash}$ with respect to $E$, 

$$\mathcal{F}_{[\omega_X]}^E(Z)\coloneqq\displaystyle\int_{X\backslash E}s(\omega)f_{\omega}^Z\displaystyle\frac{\omega^n}{n!},$$

where $f_{\omega}^Z$ is the holomorphy potential for $Z$ with respect to $\omega$. This quantity is independent of the choice of representative in the given K\"ahler class. For $Z_f, Z_g\in h_{E}^{\slash\slash}$ normalized to $\displaystyle\int_{X\backslash E}f_{\omega}^Z \displaystyle\frac{\omega^n}{n!}=\displaystyle\int_{X\backslash E}g_{\omega}^Z \displaystyle\frac{\omega^n}{n!}=0$, we define the following inner product

$$\langle Z_f, Z_g\rangle\coloneqq\displaystyle\int_{X\backslash E}f_{\omega}^Z g_{\omega}^Z\displaystyle\frac{\omega^n}{n!}.$$

 Take $T\subset K$ a compact torus whose Lie algebra is $\mathfrak{t}$. Then we see $\langle\cdot, \cdot\rangle$ is positive definite on $\mathfrak{t}$. The extremal vector field is then defined by the duality of the Futaki character under this inner product.

\medskip

Note with similar arguments as in the cscK case, and we know that $\nabla^{1,0}s(\omega_{sol})$, when restricted to $E$, becomes the extremal vector field $\nabla^{1,0}s(\omega_{FS})$ on $E$, which equals to zero. On the other hand, from Equation~\eqref{modified extK equation}, we decompose $\nabla^{1,0}s(\omega_{sol})$ into two parts. For the parts involving $Y$, which lives in the space of holomorphy potential, we know from the condition (\ref{theorem condition}) that $\nabla^{1,0}(\ep^4Y)|_E=0$. For the $\mathcal D^*_{\omega_{HS}}\mathcal D_{\omega_{HS}}(\psi_{f_{\ep}})$ term, it gives nonzero terms when restricted to $E$. Hence it has to vanish, and the metric is indeed extremal.

\printbibliography
\addcontentsline{toc}{section}{References}
\Addresses
\end{document}